\documentclass[graybox]{svmult}

\usepackage{mathptmx}       % selects Times Roman as basic font
\usepackage{helvet}         % selects Helvetica as sans-serif font
\usepackage{courier}        % selects Courier as typewriter font
\usepackage{type1cm}        % activate if the above 3 fonts are
\usepackage{makeidx}         % allows index generation
\usepackage{graphicx, lipsum}        % standard LaTeX graphics tool
\usepackage{multicol}        % used for the two-column index
\usepackage[bottom]{footmisc}% places footnotes at page bottom

\usepackage{wrapfig}
\usepackage{fancybox}

\usepackage[notcite,notref,final]{showkeys}
\usepackage{latexsym,amsxtra,amscd,amsmath,color}
\usepackage{amsfonts}
\usepackage{amssymb}
\usepackage{enumerate}
 \usepackage{tikz}
  \usepackage{harpoon}
\usetikzlibrary{matrix,arrows}

\usepackage[font=small,justification=centering]{caption}

\newenvironment{psmallmatrix}
  {\left(\begin{smallmatrix}}
  {\end{smallmatrix}\right)}

\newcommand{\red}{\textcolor{red}}

\renewcommand{\arraystretch}{1.3}

\makeindex             % used for the subject index
\begin{document}

\title*{An Invitation to Noncommutative Algebra}
% Use \titlerunning{Short Title} for an abbreviated version of
% your contribution title if the original one is too long
\author{Chelsea Walton}
% Use \authorrunning{Short Title} for an abbreviated version of
% your contribution title if the original one is too long
\institute{Chelsea Walton \at The University of Illinois at Urbana-Champaign, Department of Mathematics,
273 Altgeld Hall,
1409 W. Green Street,
Urbana, IL 61801,  \email{notlaw@illinois.edu}}

%
% Use the package "url.sty" to avoid
% problems with special characters
% used in your e-mail or web address
%
\maketitle

\abstract*{\red{[FINISH- aimed general mathematical audience- one semester of undergraduate algebras (covering groups and a tad bit of rings) would be enough].....................................}}

\abstract{This is a brief introduction to the world of Noncommutative Algebra aimed at advanced undergraduate and beginning graduate students.
}

%%%%%%%%%%%%%%%%%%%%%%%%%%%%%%%%%%
%%%%%%%%%%%%%%%%%%%%%%%%%%%%%%%%%%
%%%%%%%%%%%%%%%%%%%%%%%%%%%%%%%%%%

\section{Introduction}
\label{sec:intro}

%\red{[We work over fields of characteristic 0, unless started otherwise.]}

The purpose of this note is to invite you, the reader, into the world of Noncommutative Algebra. What is it? In short, it is the study of algebraic structures that have a noncommutative multiplication. One's first encounter with these structures occurs typically with matrices. Indeed, given two $n$-by-$n$ matrices $X$ and $Y$ with $n > 1$, we get that $X Y \neq Y X$ in general. But this simple observation motivates a deeper reason why Noncommutative Algebra is ubiquitous...

\smallskip
Let's consider two basic transformations of images in real 2-space: Rotation by 90 degrees clockwise and Reflection about the vertical axis. As we see in the figures below, the {\it order} in which these transformations are performed {\it matters}.
%namely, these transformations do not commute under composition.

\begin{figure}[h]
\vspace{-.15in}
\centerline{\includegraphics[width=0.82\textwidth]{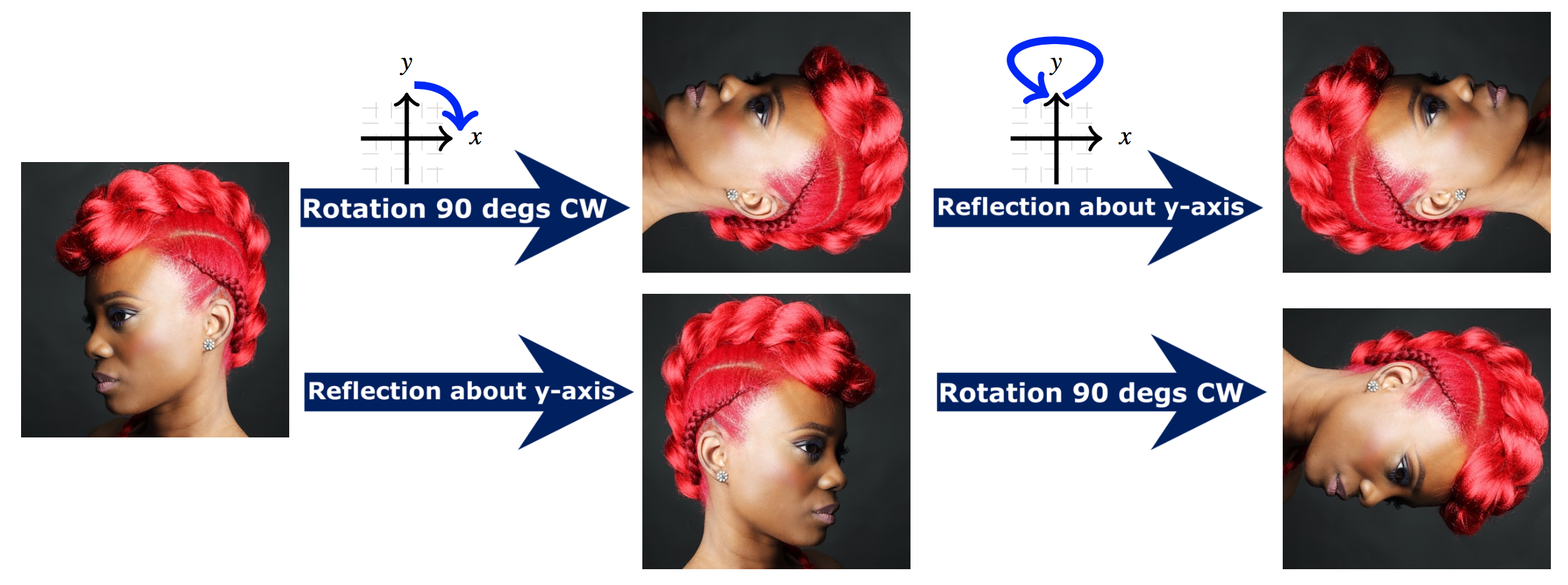}}
\vspace{-.1in}
\small{\caption{The composition of rotation and reflection transformations is noncommutative. }}
\end{figure}

\pagebreak

\noindent Since these transformations are {\it linear} (i.e., in $\mathbb{R}^2$, lines are sent to lines), they can be encoded by 2-by-2 matrices with entries in $\mathbb{R}$ \cite[Section~3.C]{Axler}. Namely, 
\begin{itemize} 
\item $90^\circ$ CW Rotation corresponds to $\begin{psmallmatrix} 0 & 1 \\ -1 & 0 \end{psmallmatrix}$, which sends vector  $\begin{psmallmatrix} v_1  \\ v_2  \end{psmallmatrix}$ to $\begin{psmallmatrix} v_2  \\ -v_1  \end{psmallmatrix}$;
\medskip
\item Reflection about the $y$-axis  is encoded by $\begin{psmallmatrix} -1 & 0 \\ 0 & 1 \end{psmallmatrix}$, which sends $\begin{psmallmatrix} v_1  \\ v_2  \end{psmallmatrix}$ to $\begin{psmallmatrix} - v_1  \\ v_2  \end{psmallmatrix}$.
\end{itemize}
The composition of linear transformations is then encoded by matrix multiplication. So, the  first row in Figure 1 is corresponds to $\begin{psmallmatrix} -1 & 0 \\ 0 & 1 \end{psmallmatrix}$$\begin{psmallmatrix} 0 & 1 \\ -1 & 0 \end{psmallmatrix}$ = $\begin{psmallmatrix} 0 & -1 \\ -1 & 0 \end{psmallmatrix}$, which sends $\begin{psmallmatrix} v_1  \\ v_2  \end{psmallmatrix}$ to $\begin{psmallmatrix} - v_2  \\ -v_1  \end{psmallmatrix}$. Yet the second row is given by  $\begin{psmallmatrix} 0 & 1 \\ -1 & 0 \end{psmallmatrix}$$\begin{psmallmatrix} -1 & 0 \\ 0 & 1 \end{psmallmatrix}$ = $\begin{psmallmatrix} 0 & 1 \\ 1 & 0 \end{psmallmatrix}$, sending $\begin{psmallmatrix} v_1  \\ v_2  \end{psmallmatrix}$ to $\begin{psmallmatrix} v_2  \\ v_1  \end{psmallmatrix}$. Therefore, the outcome of Figure 1 is a result of  the fact that $\begin{psmallmatrix} 0 & -1 \\ -1 & 0 \end{psmallmatrix}$ $\neq$  $\begin{psmallmatrix} 0 & 1 \\ 1 & 0 \end{psmallmatrix}$.

\smallskip
One can cook up other, say higher dimensional, examples of the varying outcomes of composing linear transformations by exploiting the noncommutativity of matrix multiplication. This is all part of the general phenomenon that {\it functions} do not commute under composition  typically. (Think of the myriad of outcomes of composing functions from everyday life-- for instance, washing and drying clothes!) 

\bigskip

Now let's turn our attention to special functions that we first encounter as children: {\bf Symmetries}. To make this concept more concrete mathematically, consider the informal definition and notation below.

\begin{definition} \label{def:sym} Take any object $X$. Then, a {\it symmetry} of $X$ is an invertible,
 structure/ property-preserving transformation from $X$ to itself. The collection of such transformations is denoted by Sym($X$). 
\end{definition}

Historically, the examination of symmetries in mathematics and physics served as one of the inspirations for the defining a {\it group} as an abstract algebraic structure (see, e.g., \cite[Section~1(c)]{Kleiner}). Namely, Sym($X$) is a group with the identity element $e$ being the ``do nothing" transformation, with composition as the associative binary operation, and Sym($X$) is equipped with inverse elements by definition. 

\smallskip
Continuing the example above: Take $X = \mathbb{R}^2$ and  Sym($\mathbb{R}^2$) to be the collection of $\mathbb{R}$-linear transformations from $\mathbb{R}^2$ to $\mathbb{R}^2$ (so the origin is fixed). We get that Sym($\mathbb{R}^2$) is the {\it general linear group} GL($\mathbb{R}^2$), often written as GL$_2$($\mathbb{R}$) denoting the group of all invertible 2-by-2 matrices with real entries. Further, this group is {\it nonabelian}; thus, composition of $\mathbb{R}$-linear symmetries of $\mathbb{R}^2$ is~noncommutative.

\bigskip

Another concept that is inherently noncommutative is that of a {\bf representation}. We will see later in Section~3 that this is best motivated by elementary problem of finding {\it matrix solutions to equations} (which, in turn, can have physical implications).  But for now let's think about the problem below.

\begin{question} \label{ques:x2=1}
Which matrices $M \in$ Mat$_2(\mathbb{R})$ satisfy the equation $x^2 = 1$?
\end{question}

Now  one could do the chore of writing down an arbitrary matrix $M =\begin{psmallmatrix} a & b \\ c & d \end{psmallmatrix}$ and solve for entries $a,b,c,d$ that satisfy $$\begin{pmatrix} a & b \\ c & d \end{pmatrix}\begin{pmatrix} a & b \\ c & d \end{pmatrix}= \begin{pmatrix} 1 & 0 \\ 0 & 1 \end{pmatrix}.$$ Not only is this boring, and it can be very tedious to find solutions to more general problems (e.g., taking instead $M \in$ Mat$_n(\mathbb{R})$ for $n>2$). For a more elegant approach to Question~\ref{ques:x2=1}, consider an abstract algebraic structure $T$ defined by the equation $x^2 =1$, and link $T$  to Sym($\mathbb{R}^2$) via a structure preserving map $\phi$. Then, a solution to Question~\ref{ques:x2=1} is produced in terms of an image of  $\phi$.

\smallskip
For example, we could take $T$ to be the group $\mathbb{Z}_2$ as its presentation is given by $\langle x ~|~ x^2 = e\rangle$. An example of a structure-preserving map $\phi$ is given by
\begin{center}
$\phi: \mathbb{Z}_2 \to \text{Sym}(\mathbb{R}^2), \quad
e \mapsto \text{\{Do Nothing\}}, \quad  x \mapsto \text{\{Reflection about $y$-axis\}}$.
\end{center}
Indeed,  $\phi(gg') = \phi(g) \circ \phi(g')$ for all $g,g' \in \mathbb{Z}_2$. For instance,
\[
\phi(x) \circ  \phi(x) = \text{\{Ref. about $y$-axis\}$\circ$\{Ref. about $y$-axis\}}
 = \text{\{Nothing\}} ~= \phi(e) ~= \phi(x^2).
\]

Since $\phi(e)$ and $\phi(x)$ correspond respectively to $\begin{psmallmatrix} 1 & 0 \\ 0 & 1 \end{psmallmatrix}$ and $\begin{psmallmatrix} -1 & 0 \\ 0 & 1 \end{psmallmatrix}$, these matrices are solutions to Question~\ref{ques:x2=1}. Further, other reflections of $\mathbb{R}^2$ produce additional  solutions to Question~\ref{ques:x2=1}. (Think about if {\it all} solutions to Question~\ref{ques:x2=1} can be constructed in this manner.)

\vspace{-.15in}

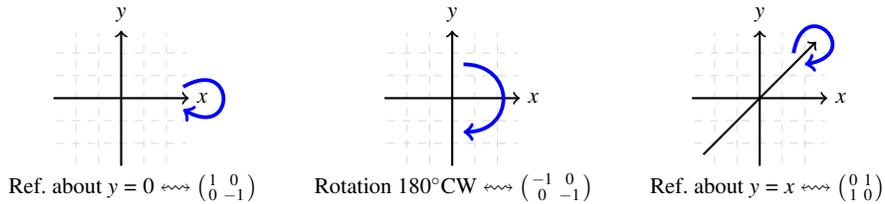
\begin{figure}[h]
\hspace{.2in}
\begin{tikzpicture}[scale=0.3]
\draw[help lines, color=gray!30, dashed] (-2.9,-2.9) grid (2.9,2.9);
\draw[->, thick] (-3,0)--(3,0) node[right]{{$x$}};
\draw[->,thick] (0,-3)--(0,3) node[above]{{$y$}};
\draw[->,line width=.5mm, blue] (2.75,.5) to [out=30,in=330, looseness = 7] (2.75,-.5);
\end{tikzpicture}
\hspace{.7in}
\begin{tikzpicture}[scale=0.3]
\draw[help lines, color=gray!30, dashed] (-2.9,-2.9) grid (2.9,2.9);
\draw[->,thick] (-3,0)--(3,0) node[right]{{$x$}};
\draw[->,thick] (0,-3)--(0,3) node[above]{{$y$}};
\draw[->,line width=.5mm, blue] (.5,1.5) to [out=0,in=0, looseness = 2] (.5,-1.5);
\end{tikzpicture}
\hspace{.7in}
\begin{tikzpicture}[scale=0.3]
\draw[help lines, color=gray!30, dashed] (-2.9,-2.9) grid (2.9,2.9);
\draw[->,thick] (-3,0)--(3,0) node[right]{{$x$}};
\draw[->,thick] (0,-3)--(0,3) node[above]{{$y$}};
\draw[->,thick] (-2.5,-2.5)--(2.5,2.5) ;
\draw[->,line width=.5mm, blue] (1.5,2) to [out=80,in=10, looseness = 10] (2,1.5);
\end{tikzpicture}

Ref. about $y$\;=\;$0$ $\leftrightsquigarrow \begin{psmallmatrix} 1 & 0 \\ 0 & -1 \end{psmallmatrix}$ \hspace{.24in} Rotation $180^\circ$CW $\leftrightsquigarrow \begin{psmallmatrix} -1 & 0\\ 0 & -1 \end{psmallmatrix}$ \hspace{.24in} Ref. about  $y$\;=\;$x$  $\leftrightsquigarrow \begin{psmallmatrix} 0 & 1 \\ 1 & 0 \end{psmallmatrix}$
{\small \caption{Reflections of   $\mathbb{R}^2$
 and the corresponding solution to Question~\ref{ques:x2=1}.} }
\end{figure}

\vspace{-.15in}

Continuing this example, instead of using the abstract group $\mathbb{Z}_2$ we could have used the  {\it group algebra} $T = \mathbb{R} \mathbb{Z}_2$, as it encodes the same information needed to address Question~\ref{ques:x2=1}. We will chat more about abstract algebraic structures in Section~2 (see Figure~5); in any case, their representations are defined informally below.

\begin{definition} \label{def:rep} Given an abstract algebraic structure $T$, we say that a {\it representation} of $T$ is an object $X$ equipped with a structure/ property-preserving map $T \to \text{Sym}(X)$.
\end{definition}

An example of a representation of a group $G$ is a vector space $V$ equipped with a group homomorphism $G \to GL(V)$, where $GL(V)$ is the group of invertible linear transformations from $V$ to itself (e.g., $GL(\mathbb{R}^2) = GL_2(\mathbb{R})$ as discussed above). Just as a representation of $G$ is identified as a {\it $G$-module}, representations of rings and of algebras coincide with {\it modules} over such structures (see Figure~12 below). See also \cite[Chapters~1 and~3]{Lorenz} for further reading and examples. 

\smallskip
Now {\it Representation Theory} is essentially a noncommutative area due to the following key fact. Take $A$ to a be commutative algebra over a field $\Bbbk$ with a representation $V$ of $A$, that is, a $\Bbbk$-vector space $V$ equipped with algebra map $\phi: A \to GL(V)$. If $(V,\phi)$ is  {\it irreducible} [Definition~\ref{def:rep-prop}], then $\dim_\Bbbk V = 1$ \cite[Section~1.3.2]{Lorenz}. Therefore, representations of commutative algebras aren't so interesting.

\smallskip
Moreover, Representation Theory is a vital subject because the problem of finding matrix solutions to equations is quite natural. Since this boils down down to studying representations of algebras that are generally noncommutative, the ubiquity of Noncommutative Algebra is conceivable. (Equations that correspond to representations of groups, like in Question~\ref{ques:x2=1}, are special.)

\bigskip

To introduce the final notion in Noncommutative Algebra that we will highlight in this paper, observe symmetries and representations both occur under an {\it action} of a gadget $T$ on an object $X$, but the difference is that symmetries form the gadget $T$ (what is acting on an object), whereas representations are considered to be the object $X$ (something being acted upon). What happens to these notions if we consider {\bf deformations} of $T$ and $X$? Consider the following informal terminology.

\begin{definition} \label{def:deform} A {\it deformation of an object} $X$ is an object $X_{\text{def}}$ that has many of the same characteristics of $X$, possibly with the exception of a few key features. In particular, a {\it deformation of an algebraic structure} $T$ is an algebraic structure $T_{\text{def}}$ of the same type that shares a (less complex) underlying  structure of $T$.
\end{definition}

For example, a deformation of a ring $R$ could be another ring $R_{\text{def}}$ that equals $R$ as abelian groups, possibly with a different multiplication than that of $R$ (see Figure~5).

\smallskip
 Now if we deform an object $X$, is there a gadget $T_{\text{def}}$ that acts on $X_{\text{def}}$ naturally? On the other hand, if we deform the gadget $T$, is there a natural deformation $X_{\text{def}}$ of $X$ that comes equipped with an action of $T_{\text{def}}$? These are obvious questions, yet the answers are difficult to visualize. 
This is because, visually, symmetries of an object $X$ are destroyed when $X$ is altered, even slightly; see Figure~3 below.

\vspace{.2in}

\noindent  \hspace{.2in} {\small $X$}: \hspace{.4in} {\small Equilateral triangle \hspace{.5in}
Isosceles triangle \hspace{.5in}
Scalene triangle} \\

\vspace{-.4in}

\begin{figure}
 \hspace{.9in} \includegraphics[width=0.15\textwidth]{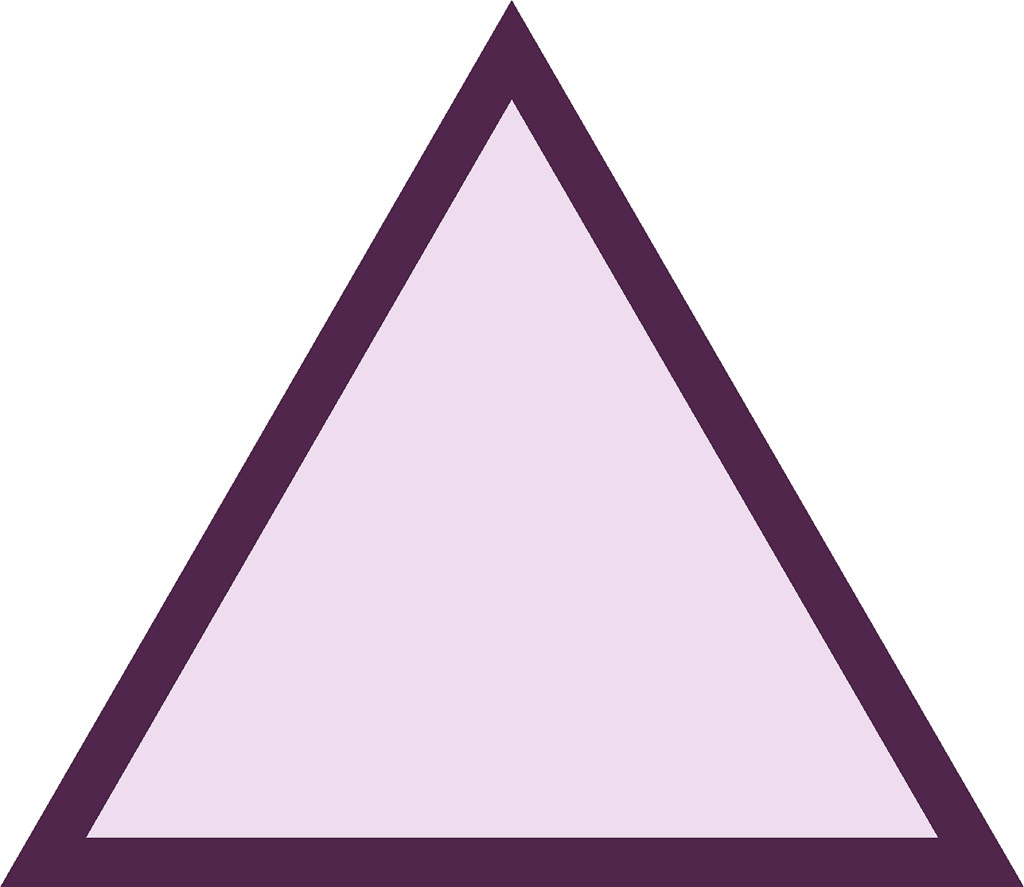}
 \hspace{.7in} \includegraphics[width=0.15\textwidth]{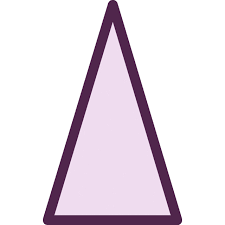}
 \hspace{.6in} \includegraphics[width=0.15\textwidth]{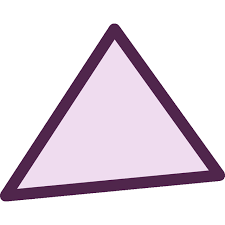}\\
Sym($X$): \hspace{.2in} Dihedral group of order 6 \hspace{.2in}
Cyclic group of order 2 \hspace{.2in}
Trivial group of order 1
{\small \caption{Triangles and their respective symmetry groups}}
\end{figure}

\vspace{-.1in}

 So we need to think beyond what can be visualized and consider a larger mathematical framework that handles symmetries under  deformation. To do so, it is essential to think beyond group actions, because, many classes of groups, including finite groups, do not admit deformations. However, {\it group algebras} or {\it function algebras on groups} do admit deformations, so we include these gadgets in the improved framework to study symmetries. We will see later in Section~4 that when symmetries are recast in the setting where they could be preserved under deformation,  other interesting and more general algebraic gadgets like {\it bialgebras} and {\it Hopf algebras} arise in the process. This is crucial in Noncommutative Algebra as some of the most important rings, especially those arising in physics, are noncommutative deformations of commutative rings; the symmetries of such deformations deserve attention.

\medskip

\begin{center}
\begin{tabular}{c}\;\; \textit{{\bf Symmetries}, {\bf Representations}, and {\bf Deformations}  will play a key role}\\  \textit{throughout this article, just as they do in Noncommutative Algebra.}\end{tabular}
\end{center}

\medskip

The remainder of the paper is two-fold: first, we will review three historical snapshots of how Noncommutative Algebra played a prominent role in mathematics and physics. We will discuss William Rowen Hamilton's Quaternions in Section~2 and the birth of Quantum Mechanics in Section~3. We will also briefly discuss the emergence of Quantum Groups in Section~4, together with the concept of Quantum Symmetry.  In Section~5 we present a couple of research avenues for further investigation. All of the material here is by no means exhaustive, and many references will be provided throughout.

%%%%%%%%%%%%%%%%%%%%%%%%%%%%%%%%%%
%%%%%%%%%%%%%%%%%%%%%%%%%%%%%%%%%%
%%%%%%%%%%%%%%%%%%%%%%%%%%%%%%%%%%

\section{Hamilton's Quaternions (1840s - 1860s)}
\label{sec:hamilton}

Can {\it numbers} be noncommutative? The best answer is,  as always, ``Sure, why not?" 

\begin{wrapfigure}{r}{0.25\textwidth}
\vspace{-.3in}
  \begin{center}
    \includegraphics[trim=0cm 0cm 0cm 0cm,clip=true,width=3.3cm]{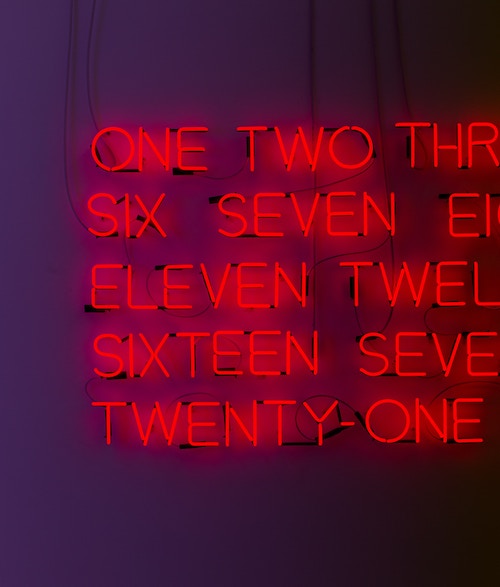}
  \end{center}
 {\small  \caption{Numbers that we all know and love... but we should love more! }}
  \vspace{-.5in}
\end{wrapfigure}
In this section we will explore a {number system that generalizes both the systems of real numbers $\mathbb{R}$ and complex numbers $\mathbb{C}$. The key feature of this new collection of numbers --the {\it quaternions} $\mathbb{H}$--  is that they have a noncommutative multiplication! This feature caused a bit of ruckus for William Rowan Hamilton (1805-1865) after his discovery of the quaternions in the mid-19th century.

\medskip

\begin{quote}
``Quaternions came from Hamilton after really good work had been done; and, though beautifully ingenious, have been an unmixed evil to those who have touched them in any way..."

-- Lord Kelvin, 1892
\end{quote}

\medskip

Now what do we mean by a {\it number system}? Loosely speaking, it is a {\it set of quantities} used to measure or count (a collection of) objects, which is equipped with an {\it algebraic structure} [Figure~5]. 

\begin{figure}[h] 
\centering \includegraphics[width=0.8\textwidth]{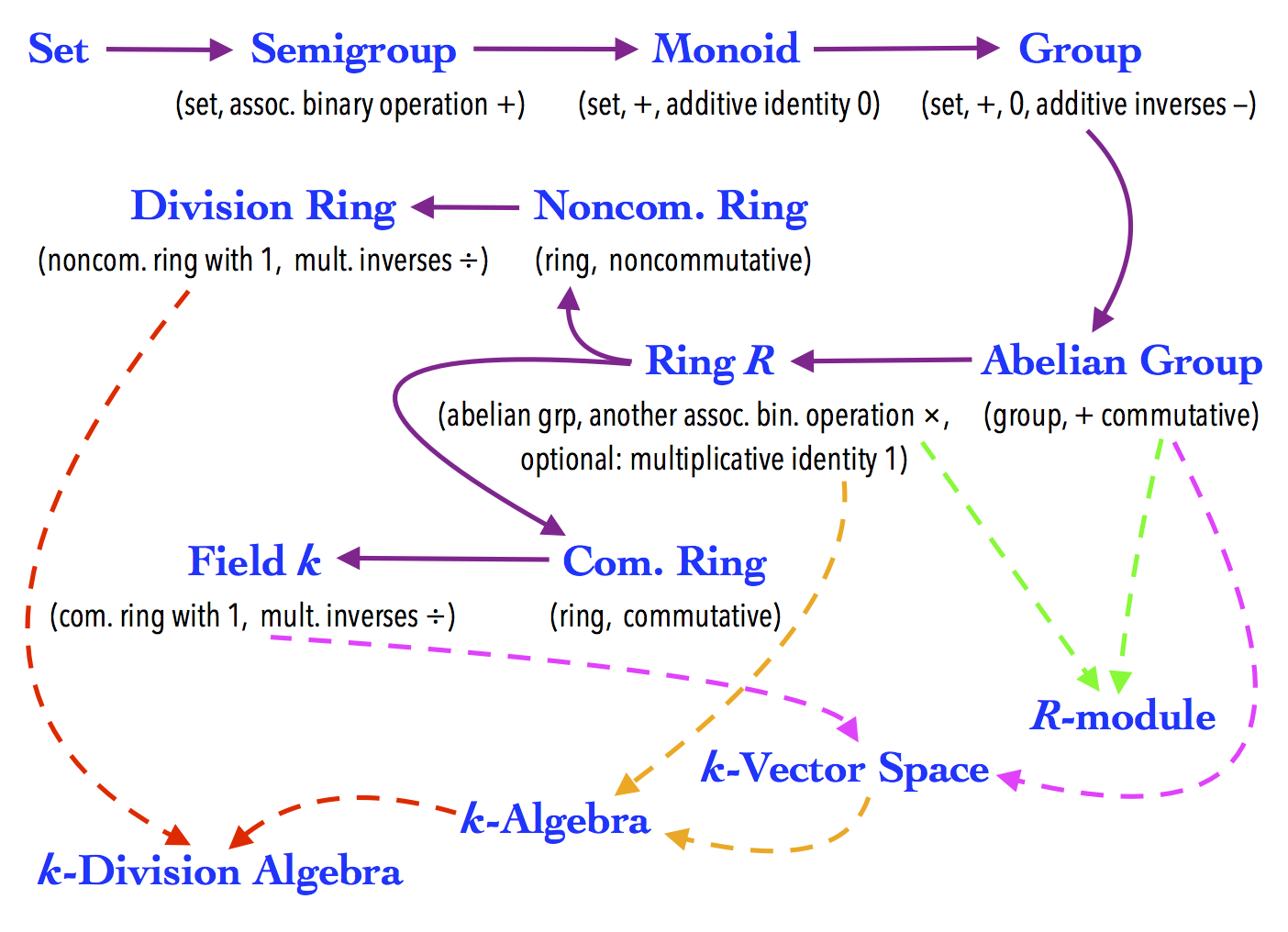}
{\small \caption{Some algebraic structures. Straight arrows denote structures increasing in complexity. Dashed arrows denote structures merging compatibly to form another structure.}} \label{fig:algstr}
\vspace{-.15in}
\end{figure}

Since we should be able to add, subtract, multiply and divide numbers, we consider the following terminology.

\begin{definition} \label{def:div} Fix $n \in \mathbb{Z}_{\geq 1}$.
An {\it $n$-dimensional division algebra $D$ over $\mathbb{R}$} consists of the set of $n$-tuples of real numbers $\underline{a}:=(a_1, a_2, \dots, a_n)$, $a_i \in \mathbb{R}$, with ${\bf 0}:=(0, 0, \dots, 0)$ and a unique element designated as ${\bf 1}$ so that 
\begin{itemize} 
\item we can add and subtract two $n$-tuples $\underline{a}$ and $\underline{b}$ component-wise to form $\underline{a + b}$ and $\underline{a - b}$ in $D$, respectively;
\item we can multiply $\underline{a}$ by a scalar $\lambda \in \mathbb{R}$ component-wise to form $\underline{\lambda \ast a}$ in $D$; 
\item there is a rule for multiplying $\underline{a}$ and $\underline{b}$ to form $\underline{a \cdot b}$ in $D$ (this is not necessarily done component-wise, nor does it need to be commutative); and
\item there is a rule for dividing $\underline{a}$ by $\underline{b} \neq {\bf 0}$ to form $\underline{a \div b}$ in $D$;
\end{itemize}
in such a way that 
\begin{enumerate}
\item[(i)] $(D, +, -, {\bf 0})$ is an {\it abelian group}, 
\item[(ii)] $(D, +, -, {\bf 0}, \ast)$ is an {\it $\mathbb{R}$-vector space}, 
\item[(iii)] $(D, \cdot, {\bf 1})$ is an associative {\it unital ring}, and 
\item[(iv)] $(D, +, -, \ast, \cdot, {\bf 0}, {\bf 1})$ is an associative $\mathbb{R}$-{\it algebra} 
\end{enumerate}
all in a compatible fashion (e.g. $\cdot$ distributes over +, etc.).
\end{definition}

As one can imagine, there are not many of these gadgets floating around as they have a {\it lot} of structure. A {\it $1$-dimensional division algebra $D$ over $\mathbb{R}$} must be the field $\mathbb{R}$ itself. Moreover, a {\it $2$-dimensional division algebra $D$ over $\mathbb{R}$} is isomorphic to the field of complex numbers $\mathbb{C}$, where the pair $(a_1, a_2)$ is identified with the element $a_1 + a_2 i$ for $i^2 = -1$. 
The algebraic structure for the pairs then follows accordingly, 
\begin{wrapfigure}{l}{0.35\textwidth}
\vspace{-.23in} 
\centering \includegraphics[width=0.3\textwidth]{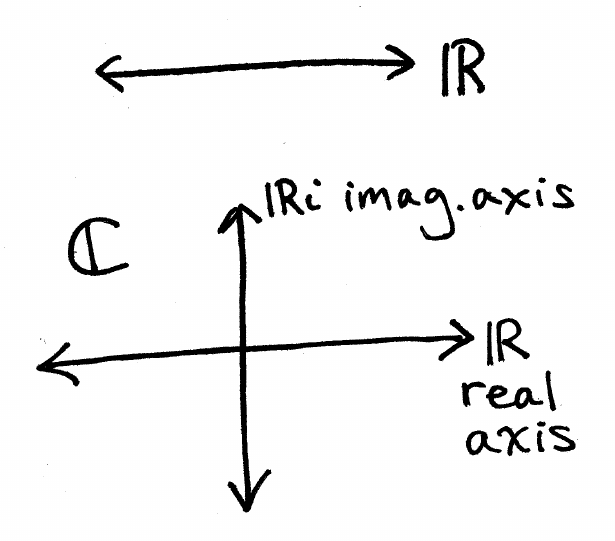}
{\small \caption{The real line, and the complex plane visualized as $\mathbb{R}^2$.}}
\vspace{-.33in} 
\end{wrapfigure}
e.g., the multiplication of $\mathbb{C}$ is given by $$(a_1, a_2) \cdot (b_1, b_2) = (a_1b_1-a_2b_2, ~ a_1b_2+a_2b_1).$$ Note that the 1- and 2- dimensional real division algebras, $\mathbb{R}$ and $\mathbb{C}$, are commutative rings, and these can be viewed geometrically as in Figure~6.

\smallskip

A natural question is then the following.

\begin{question} \label{ques:div}
What are the $n$-dimensional real division algebras for $n \geq 3$?
\end{question}

Hamilton obsessed over this question, especially the $n=3$ case, for over a decade. Even his children would routinely ask him, ``Papa, can you multiply triplets?".
\begin{wrapfigure}{r}{0.42\textwidth}
\vspace{-.23in} 
\centering \includegraphics[width=0.42\textwidth]{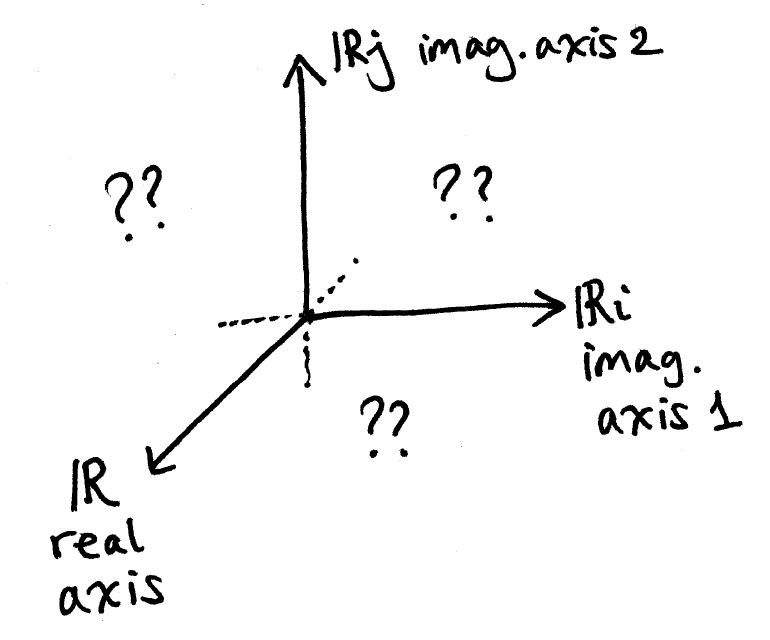}
{\small \caption{A failed attempt at a 3-dimensional number system.}}
\end{wrapfigure}
\begin{wrapfigure}{l}{0.35\textwidth}
  \begin{center}
  \vspace{-.3in} 
    \includegraphics[width=0.35\textwidth]{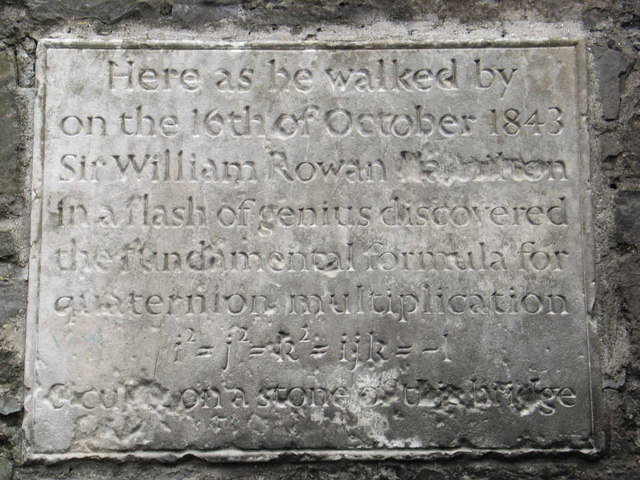}
  \end{center}
 {\small \caption{Plaque on Brougham Bridge in Dublin, recognizing Hamilton's invention}}
  \vspace{-.2in}
\end{wrapfigure}

\vspace{-.1in}
\noindent His initial ideas were to use two imaginary axes $i$ and $j$ so that the 3-tuples $(a_1, a_2, a_3)$ of a 3-dimensional number system correspond to $a_1 + a_2 i + a_3 j$.
However, he could not cook up rules that $i$ and $j$ should obey to make this collection of triples a valid division algebra \cite{Hamilton, May, vdW}. According to some mathematicians, this obsession was quite `Mad' \cite{Bayley, NPR}.

\smallskip

Finally, on October 16th 1843, on a walk with his wife in Dublin, Hamilton had a moment of Eureka!  In his words to his son Archibald,

\smallskip

\begin{quote}
``An {\it electric} circuit seemed to close; and a spark flashed forth, the herald, as I {\it foresaw immediately}, of many long years to come of definitely directed thought and work [...] Nor could I resist the impulse, unphilosopical as it may have been,  to cut with a knife on the stone of Broughham Bridge, as we passed it, the fundamental formula with the symbols $i$, $j$, $k$; namely 
$i^2 = j^2 = k^2 = ijk = -1,$ which contains the {\it solution} of the {\it problem}..."

-- W. R. Hamilton, August 5th, 1865
\end{quote}

Hamilton had discovered that day a number system generalizing both $\mathbb{R}$ and $\mathbb{C}$, consisting of {\it 4-tuples} of real numbers, not constructed from triplets as he had imagined for so long  \cite{Hamilton}.

\begin{definition}\label{def:H}
The {\it quaternions} is a 4-dimensional real division algebra, denoted by $\mathbb{H}$, comprised of 4-tuples of real numbers $\underline{a} := (a_0, a_1, a_2, a_3)$, which are identified as elements of the form 
$$a_0 + a_1i +a_2j + a_3k, \quad \text{ for } a_i \in \mathbb{R},$$
where addition, subtraction and scalar multiplication are performed component-wise, and multiplication and division are governed by the rule
$$i^2 = j^2 = k^2 = ijk = -1.$$
\end{definition}

\noindent Observe  that $jk = i$, whereas $kj = -i$. Therefore, $\mathbb{H}$ is a noncommutative ring!

\pagebreak

In any case, notice that the multiplication rule of $\mathbb{H}$ is a bit complicated:

\begin{equation} \label{eq:long}
(a_0, a_1, a_2, a_3) \cdot (b_0, b_1, b_2, b_3) = 
{\small \left(\begin{array}{c} 
a_0b_0 - a_1b_1 - a_2b_2 - a_3b_3,\\
a_0b_1 +a_1b_0 +  a_2b_3 - a_3b_2,\\
a_0b_2 - a_1b_3 + a_2b_0 +  a_3b_1 ,\\
a_0b_3 + a_1b_2 - a_2b_1 + a_3b_0  
\end{array} \right)}
\end{equation}

\noindent... and let's not commit this rule to memory. To circumvent this issue  Hamilton gave the quaternions a geometric realization that encodes their multiplication. Namely, for $\underline{a} := a_0 + a_1i +a_2j + a_3k \in \mathbb{H}$, let
\begin{center}
\begin{tabular}{l}
$a_0$  \; be the ``scalar" component of $\underline{a}$, and \\
 $\overrightharp{a} :=a_1i +a_2j + a_3k$ \; be the ``vector" component of $\underline{a}$. 
\end{tabular}
\end{center}
Then, the vector components are visualized as points/ vectors in $\mathbb{R}^3$, whereas the scalar component is realized as an element of {\it time}. 
See, for instance, the footnote on page~60 and other parts of the preface of \cite{HamiltonLectures} for Hamilton's original thoughts on the connection between the quaternions and the laws of space and  time.  (Yes, yes, this was all very controversial back then!)

\smallskip
Hamilton then devised two vector operations, now known as the {\it dot product} ($\bullet$) and {\it cross product} ($\times$) to make the multiplication rule of $\mathbb{H}$ more compact:

\begin{equation} \label{eq:short}
 \underline{a} \cdot \underline{b} = \left[a_0b_0 ~- ~\overrightharp{a} \bullet \overrightharp{b}\right] + \left[a_0 \overrightharp{b} ~+ ~b_0 \overrightharp{a} ~+~ \overrightharp{a} \times \overrightharp{b}\right], \quad \quad \forall \underline{a}, \underline{b} \in \mathbb{H}.
\end{equation}
Not only is formula \eqref{eq:short} easier to retain than \eqref{eq:long}, the (commutative) dot product and (noncommutative) cross product have appeared  in various parts of mathematics
\begin{wrapfigure}{l}{0.4\textwidth}
\vspace{-.25in}
\centering \includegraphics[width=0.4\textwidth]{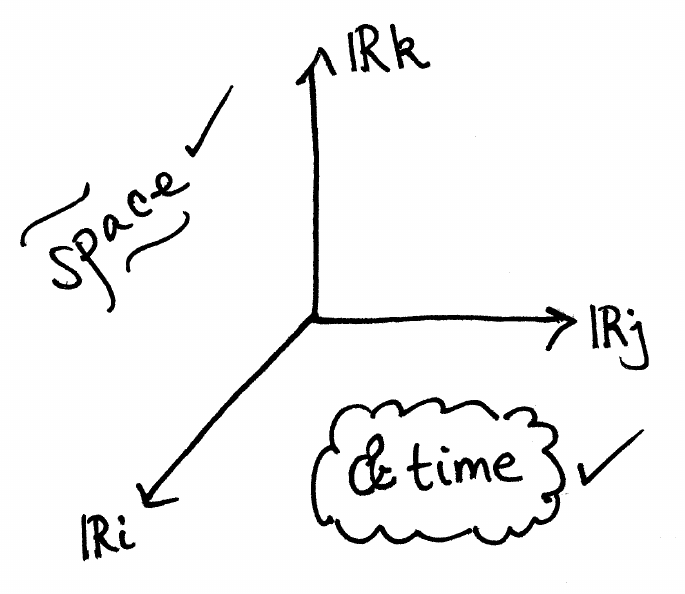}
{\small \caption{A successful attempt at a 4-dimensional number system.}}
\vspace{-.2in}
\end{wrapfigure}
  and physics throughout the years, including our multi-variable calculus courses.

\smallskip
Geometrically, the operations in $\mathbb{H}$ capture symmetries of $\mathbb{R}^3$ [Definition~\ref{def:sym}]: Addition/ substraction, scalar multiplication, and multiplication/ division correspond respectively to translation, dilation, and rotation of vectors of $\mathbb{R}^3$; see, e.g., \cite[page ~272]{HamiltonLectures} and \cite{Kuipers} for a discussion of rotation. To see rotations in action, first note that the  {\it length} of a quaternion \underline{$a$} is given by 
$$|\underline{a}| := \sqrt{a_0^2 + a_1^2 +a_2^2 + a_3^2}.$$
Next, fix an {\it axis of rotation} $\overrightharp{n} :=n_1i +n_2j + n_3k$ with $|\overrightharp{n}| = 1$, a quaternion of unit-length. Then, rotating a vector $\overrightharp{q}$ about the axis $\overrightharp{n}$ clockwise by $\theta$ radians (when viewed from the origin) corresponds to {\it conjugating} $\overrightharp{q}$ by the quaternion $e^{{\frac{\theta}{2}}\overrightharp{n}}$. It's helpful to use here an extension of Euler's formula, $e^{{\frac{\theta}{2}}\overrightharp{n}} = \cos(\textstyle\frac{\theta}{2}) + \sin(\frac{\theta}{2})\overrightharp{n}$,  to understand the quaternion $e^{{\frac{\theta}{2}}\overrightharp{n}}$. An example is given in Figure~10 below.

\begin{figure}
\includegraphics[width=\textwidth]{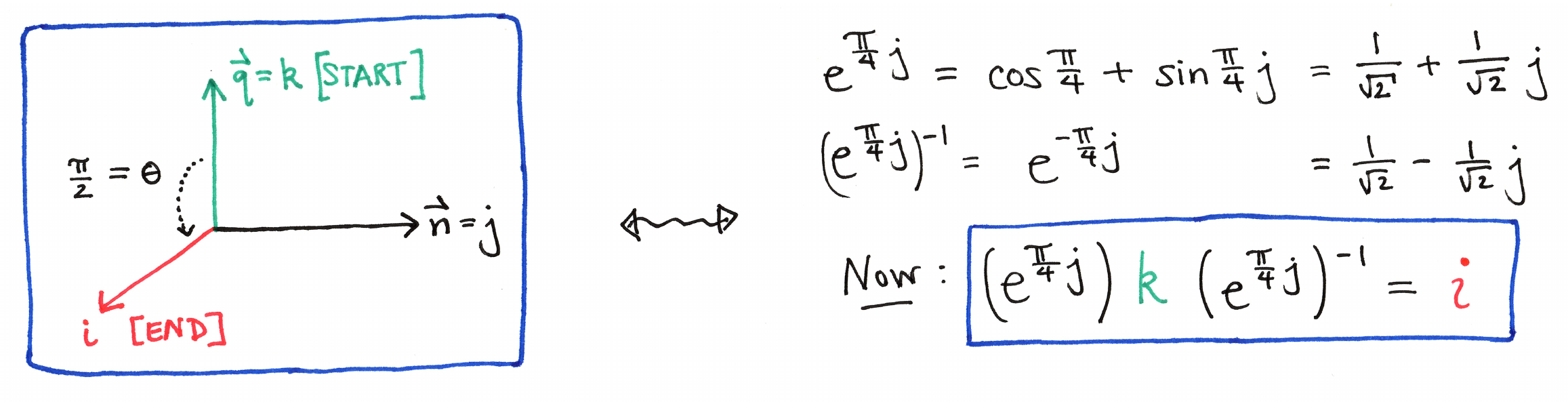}
{\small \caption{Rotating vector $k$ about axis $j$ by $\frac{\pi}{2}$ radians $\leftrightsquigarrow$ Conjugating $k$ by quaternion $e^{{\frac{\pi}{4}}j}$.}}
\end{figure}

 Moreover, rotations of $\mathbb{R}^3$ can be encoded as  a representation [Definition~\ref{def:rep}] of 
 the multiplicative subgroup $U(\mathbb{H})$ consisting of unit-length quaternions. Indeed, we have a group homomorphism 
$$\quad \quad U(\mathbb{H}) \longrightarrow GL(\mathbb{R}^3) = GL_3(\mathbb{R}), \quad \text{ given by}$$
{\small
$$\begin{array}{c}
\smallskip
a_0 + a_1 i + a_2 j + a_3 k\\
 \text{ with } |\underline{a}| = 1
 \end{array} \quad \mapsto  \quad
\begin{pmatrix} 
\smallskip
 a_0^2+a_1^2-a_2^2-a_3^2& \;\; \quad 2 a_1 a_2 - 2 a_0 a_3\; \; \quad &  2 a_1 a_3 + 2 a_0 a_2\\
\smallskip
 2 a_1 a_2 + 2 a_0 a_3 &  a_0^2-a_1^2+a_2^2-a_3^2 & 2 a_2 a_3 - 2 a_0 a_1\\
 2 a_1 a_3 - 2 a_0 a_2 & 2 a_2 a_3 + 2 a_0 a_1 &  a_0^2-a_1^2-a_2^2+a_3^2
 \end{pmatrix}.$$
 }

This geometric realization of $\mathbb{H}$ has many modern applications-- we  refer  to the text \cite{Kuipers:text} for a nice self-contained discussion of applications to computer-aided design, aerospace engineering, and other fields.

\smallskip
Returning to Question~\ref{ques:div},  its answer is now given below.

\begin{theorem}  \textnormal{\cite{Frobenius}, \cite{Palais}, \cite[Theorem~13.12]{Lam}, \cite{BM}, \cite{Kervaire}, \cite{Zorn}}
The answer to Question~\ref{ques:div} is Yes if and only if $n = 1, 2, 4, 8$. Such division algebras $D$ are unique up to isomorphism in their dimension with isomorphism class represented by 

\hspace{-.1in}
\begin{tabular}{llllll}
$\bullet$ &the real numbers $\mathbb{R}$  \;  &for $n=1$, \quad \quad\quad \quad  &
$\bullet$ &the complex numbers $\mathbb{C}$ \quad & for $n=2$, \\
$\bullet$ &quaternions $\mathbb{H}$  &for $n=4$, &
$\bullet$ &octonions $\mathbb{O}$  &for $n=8$.
\end{tabular} 

\smallskip

\noindent Here, $D$ is commutative only when $n=1,2$, and is associative only when $n=1,2,4$. 
\end{theorem}

So, Hamilton discovered the last associative finite-dimensional real division algebra, but the price that he had to pay (at least mathematically) was the loss of commutativity. Perhaps this was not too high of a price--  we are certainly willing to lose ordering when choosing to work with $\mathbb{C}$ instead of $\mathbb{R}$. If we are also willing to part with associativity, then the octonions $\mathbb{O}$ is a perfectly suitable number system; see \cite{Baez} for more details. And, of course, there are further generalizations of number systems-- see \cite{Dickson, Lewis, Wiki-CD} to start, and go wild!

\medskip
We return to the quaternions later in Section~5.1 for a discussion of potential research directions.

%%%%%%%%%%%%%%%%%%%%%%%%%%%%%%%%%%
%%%%%%%%%%%%%%%%%%%%%%%%%%%%%%%%%%
%%%%%%%%%%%%%%%%%%%%%%%%%%%%%%%%%%

\section{The Birth of Quantum Mechanics (1920s)}
\label{sec:mechanics}

Another  period that sparked an interest in Noncommutative Algebra was  the birth of Quantum Mechanics in the 1920s.   Three of the key figures during  
%\begin{wrapfigure}{l}{0.5\textwidth}
%\vspace{-.1in}
%    \includegraphics[width=0.5\textwidth]{Pic-milky}
%  \small{   \caption{``More than anything, this photograph was really the result of a series of little accidents."   \hspace{1.9in}-- Billy Huynh,  photographer
%  \hspace{1.4in}
%  %. 
%  $\dots$ So is good mathematics!}}
%  \vspace{-.1in}
%  \end{wrapfigure} 
this time were Max Born (1882-1970), Werner Heisenberg (1901-1976), and Paul Dirac (1902-1984), who were all curious about the behavior of subatomic particles \cite{BornJ, Heisenberg, Dirac}.

\begin{figure}[h]
\vspace{-.1in}
  \centerline{\includegraphics[width=0.8\textwidth]{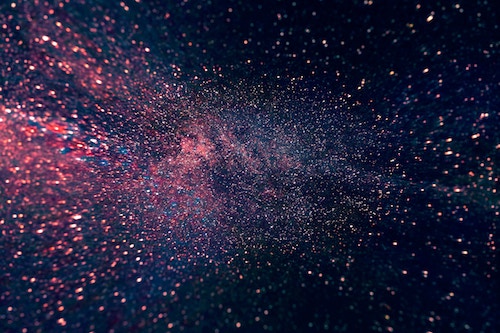}}
  \small{\caption{``More than anything, this photograph was really the result of a series of little accidents."   -- Billy Huynh,  photographer
  $\dots$ So is good mathematics!}}
  \vspace{-.1in}
  \end{figure} 
  
  Along with their colleagues, Born, Heisenberg, and Dirac believed that important aspects of subatomic behavior are those that could be {\it observed} (or measured). However, the tools of classical mechanics  available at the time (with {\it observables} corresponding to real-valued functions) were not suitable in  capturing this behavior properly. A new type of mechanics was needed, leading to the development of quantum mechanics where observables are realized as linear operators. For a great account of how this transition took place (some of which we will recall briefly below), see Part II of the  van der Waerden's text \cite{vdW67}. (For historical context of another figure, Pascual Jordan, who also played a role in these developments, see, e.g., \cite{Hentschel}.)
  %we will use this reference to skip ahead and get to the mathematics straight-away. 
 
 \smallskip
The two observables in which Born, Heisenberg, and Dirac were especially interested were the {\it position} and {\it momentum} of a subatomic particle, and they employed Niels Bohr's notion of {\it orbits} to keep track of these quantities. Mathematically, this boils down to using matrices in order to book-keep data corresponding to the observables under investigation, thus initiating {\it matrix mechanics}.
%; this prompted the launch of {\it matrix mechanics}. 
The surprising outcome of using this new matrix framework in studying subatomic particles was  stated succinctly as follows \cite{Heisenberg2}:
%\blue{These days, it is known that the best way to keep track of the data arising from observables is with matrices, yet unfortunately, matrices were not part of  physicists' toolkit in 1920s. So Heisenberg handed off his ideas to Max Born \red{[Years]} to further this line of research, most notably, his principle on the way subatomic particles behave:}

\begin{quote}
The more precisely the position is determined, the less precisely the momentum is known, and vice versa.

-- Heisenberg's ``Uncertainty Principle", 1927 
\end{quote}

More precisely, suppose that $P$ and $Q$ are square matrices of the same size representing the observables momentum and position, respectively. The fact that $P$ and $Q$ do not commute typically (as one expects in classic mechanics) led to the discovery of what Born dubbed as  {\it The Fundamental Equation of Quantum Mechanics}:
\begin{equation} \label{eq:FE}
 PQ- QP = i \hbar \ast I,
 \end{equation}
 Here $i$ is the square root of $-1$, $\hbar$ is Planck's constant, and $i \hbar \ast I$ is the scalar multiple of the identity matrix $I$ of the same size as $P$ and $Q$. For physical reasons, it was known early in the theory of quantum mechanics that matrices $P$ and $Q$ that satisfy Equation \eqref{eq:FE} should be of infinite size, and we will recall a well-known, mathematical proof of this fact later in Proposition~\ref{prop:A1-inf}.

\smallskip
As done in practice by many physicists  and mathematicians, through rescaling let's consider a {\it normalized} version of The Fundamental Equation, as this still captures the spirit of Heisenberg's Uncertainty Principle:

\begin{equation} \label{eq:NFE}
 PQ- QP =  I,
 \end{equation}
 
Now with today's technology, one convenient way of studying matrix solutions $P$ and $Q$ to Equation \eqref{eq:NFE} (or to Equation \eqref{eq:FE}) is to use the {\it theory of representations of (associative) algebras}. To see this connection, first let's fix a field $\Bbbk$ and for ease:

\medskip

\noindent 
\textbf{Standing Hypothesis}. We assume in this section that $\Bbbk$ is a field of  characteristic 0.

\medskip

 Then recall from Definition~\ref{def:div} (and Figure~5) that a {\it $\Bbbk$-algebra} $A$  is a $\Bbbk$-vector space equipped with the structure of a unital ring in a compatible fashion. In this case, $A = (A, +, -, \ast, \cdot, {\bf 0}, {\bf 1})$ where $(A, +, -, \ast, {\bf 0})$ is the $\Bbbk$-vector space structure where $+$ is the abelian group operation and $\ast$ is scalar multiplication, and $(A, \cdot, {\bf 1})$ is a unital ring with $\cdot$ denoting its multiplication. Next, we make our vague notion of representations in Definition~\ref{def:rep} more precise in the context of $\Bbbk$-algebras.

\begin{definition} Consider the following notions.
\begin{enumerate}
\item For a $\Bbbk$-vector space $V$,  the {\it endomorphism algebra} $\text{End}(V)$ on $V$ is an $\Bbbk$-algebra consisting of endomorphisms of $V$ with multiplication being composition $\circ$. (If $V$ is an $n$-dimensional $\Bbbk$-vector space, then $\text{End}(V)$ is isomorphic to the matrix algebra $\text{Mat}_n(\Bbbk)$ with matrix multiplication. Here, $n$ could be infinite.)

\smallskip

\item A {\it representation} of an associative $\Bbbk$-algebra $A$ is a $\Bbbk$-vector space $V$ equipped with a $\Bbbk$-algebra homomorphism $\phi: A \to \text{End}(V)$; say $\phi(a) =: \phi_a \in \text{End}(V)$ for $a \in A$. Namely, for all $a,b \in A$, $\lambda \in \Bbbk$, and $v \in V$, we get that
$$\phi_{a + b}(v) = \phi_{a}(v) + \phi_b(v), \quad \phi_{\lambda \ast a}(v) = \lambda \ast \phi_{a}(v), \quad \phi_{ab}(v) = (\phi_{a} \circ \phi_b)(v). $$

\item The {\it dimension} of a representation $(V, \phi)$ of an associative $\Bbbk$-algebra $A$ is the $\Bbbk$-vector space dimension of $V$, which could be infinite.
\end{enumerate}
\end{definition}

Representations of associative $\Bbbk$-algebras $A$ go hand-in-hand with $A$-modules, as illustrated in Figure~12 below.

%Indeed, if we are given a representation $(V, \phi)$ of an associative algebra $A$, then we can build a {\it (left) $A$-module} $M$ (and vice versa) as follows: $M$ is $V$ as a $\Bbbk$-vector space with structure map $\psi: A \times M \to M$ defined by $\psi(a,m) = \phi_a(m)$. 
%(later, we will denote this {\it action} by $a \rightharpoonup m$). 

\begin{figure}
\centerline{\includegraphics[width=0.7\textwidth]{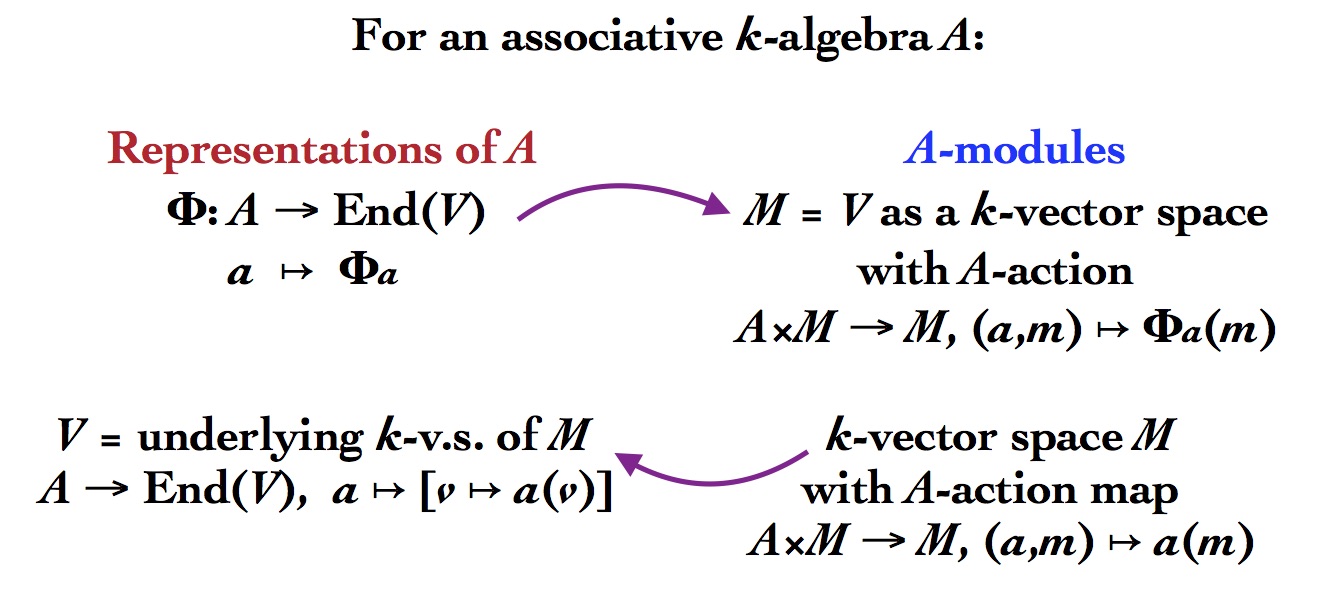}}
{\small \caption{Connection between representations and modules of $\Bbbk$-algebras.}}
\end{figure}

\smallskip
Now for the purposes of finding matrix solutions of Equation~\eqref{eq:NFE}, consider the  $\Bbbk$-algebra defined below.

\begin{definition} \label{def:A1k}
The {\it (first) Weyl algebra} over a field $\Bbbk$ is the $\Bbbk$-algebra $A_1(\Bbbk)$ generated by noncommuting variables $x$ and $y$, subject to relation $yx - xy = 1$. That is, $A_1(\Bbbk)$ has a $\Bbbk$-algebra presentation 
$$A_1(\Bbbk) = \Bbbk \langle x, y \rangle/ (yx - xy -1),$$
given as the quotient algebra of the {\it free algebra} $\Bbbk \langle x, y \rangle$ (consisting of words in variables $x$ and $y$) by the ideal $(yx - xy -1)$ of $\Bbbk \langle x, y \rangle$. 
(This algebra is sometimes referred to as the {\it Heisenberg-Weyl algebra} due to its roots in physics.)
\end{definition}

The Weyl algebra is also the first example of an {\it algebra of differential operators}-- its generators $x$ and $y$ can be viewed as the differential operators on the polynomial algebra $\Bbbk[x]$ given by multiplication by $x$ and $\frac{d}{dx}$, respectively. (Check that $\frac{d}{dx} x - x \frac{d}{dx}$ is indeed the identity operator on $\Bbbk[x]$.)

\smallskip
Returning to the problem of finding $n$-by-$n$ matrix solutions to Normalized Fundamental Equation \eqref{eq:NFE}-- this is equivalent to the task of constructing $n$-dimensional representations of $A_1(\Bbbk)$ as shown in Figure~13 below. In fact, this is why $A_1(\Bbbk)$ is known as the {\it ring of quantum mechanics}. 

%Namely, take
%\begin{equation} \label{eq:repWeyl} \phi: A_1(\Bbbk) \longrightarrow \text{Mat}_n(\Bbbk),  \quad \text{with } x \mapsto Q, \;  y \mapsto P,
%\end{equation}
%which defines an $n$-dimensional representation of $A_1(\Bbbk)$ if and only if $Q$ and $P$ satisfy %\eqref{eq:NFE}.

%\bigskip 

\begin{figure}[h]
\vspace{-.1in}
\centerline{\includegraphics[width=0.8\textwidth]{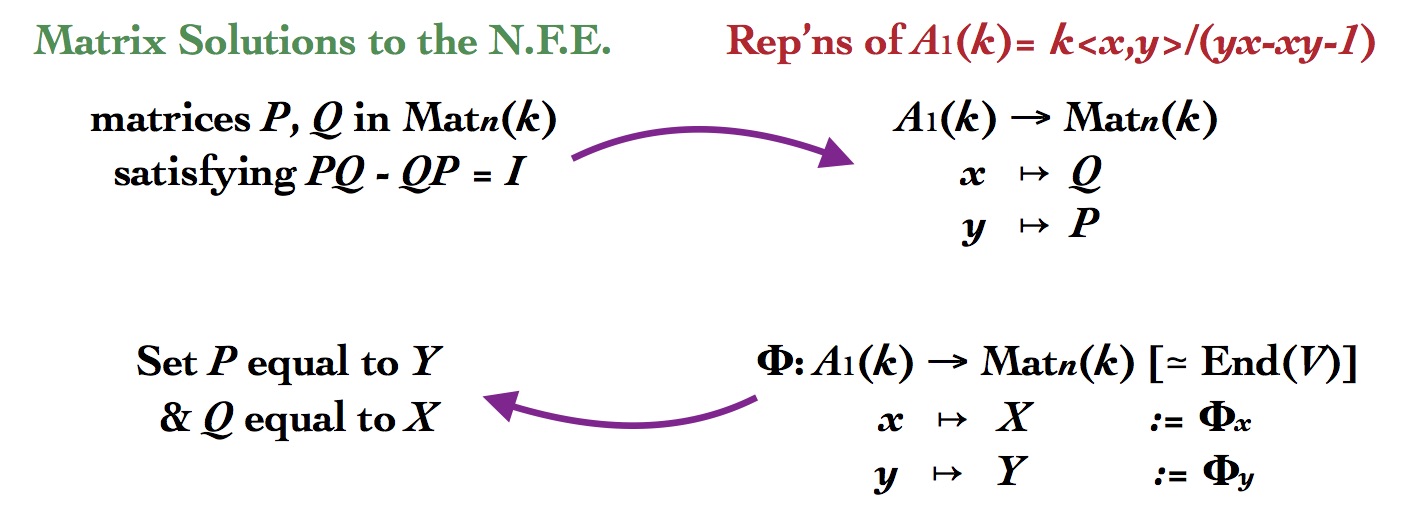}}
{\small \caption{Connection between matrix solutions to N.F.E. and representations of $A_1(\Bbbk)$.}}
\end{figure}

Next, with the toolkit of matrices handy, we obtain a well-known result on the size of matrix solutions to  \eqref{eq:NFE}. We need following facts about the {\it trace} of a square matrix $X$ (which is the sum of the diagonal entries of $X$): $\text{tr}(X \pm Y)  = \text{tr}(X) \pm \text{tr}(Y)$ and $\text{tr}(XY) = \text{tr}(YX)$ for any $X,Y\in \text{Mat}_n(\Bbbk)$.

\begin{proposition} \label{prop:A1-inf}
The Normalized Fundamental Equation \eqref{eq:NFE} does not admit finite matrix solutions, i.e., representations of $A_1(\Bbbk)$ must be infinite-dimensional.
\end{proposition}

\begin{proof}
By way of contradiction, suppose that we have matrices $P,Q \in \text{Mat}_n(\Bbbk)$ with $0<n<\infty$ so that $PQ-QP = I$. Applying trace to both sides of this equation yields:
$$ 0 ~=~ \text{tr}(PQ) - \text{tr}(PQ) ~=~ \text{tr}(PQ) - \text{tr}(QP) ~=~ \text{tr}(PQ-QP) 
~=~ \text{tr}(I) ~=~ n,$$
a contradiction as desired. \qed
\end{proof}

On the other hand, the first Weyl algebra does have an infinite-dimensional representation. Take, for instance:
{\small
\begin{equation}\label{eq:A1k-inf}
P = \arraycolsep=1.4pt\def\arraystretch{1.2} \begin{pmatrix} 
~0 ~& ~1~ &        &        &          \\ 
       & ~0~ & ~2~ &        &          \\
       &        & ~0~ & ~3~ \;\;&    \vspace{-.08in}       \\
       &        &        &~0~ & \hspace{-.1in} \ddots   \\ 
       &        &        &        & ~\ddots 
\end{pmatrix} 
\quad 
{\normalsize \text{and}}
\quad
Q = \begin{pmatrix} 
~0 ~&  &  &  &  &  \\
~1  & ~0~ & &  &  &  \\
 & ~1~  & ~0~ &  &  &  \\
 &  & ~1~ & 0~ &  &  \\
 &  &  & ~ \ddots ~ & ~\ddots&
\end{pmatrix}.
\end{equation}
}
...And there are many, many more! 

\smallskip

But finding {\it explicit} matrix solutions to equations is computationally difficult in general, especially when the most important representations of an algebra are infinite-dimensional. The power of representation theory, however, is centered on its tools to address more abstract algebraic problems that are (perhaps) related to computational goals. For instance, representation theory may address some of the following questions for a given $\Bbbk$-algebra $A$, which are all quite natural:
\begin{itemize}
\item Do representations of $A$ {\it exist}? If so, what are their {\it dimensions}?
\item When are two representations considered to be the {\it same} (or {\it isomorphic})?
\item Are (some of) the representations of $A$  {\it parametrized} by a geometric object~$\mathcal{X}$? Do isomorphism classes of representations correspond bijectively to points of~$\mathcal{X}$? 
\end{itemize}

\smallskip

We will explore a few of these questions and further notions in Section~\ref{sec:researchrep}  towards a research direction in Representation Theory. 

\smallskip

The representation theory of other algebras of differential operators have also been key in modeling subatomic behavior. This includes Dirac's {\it quantum algebra} that addresses the question of how several position observables ($Q_1, \dots, Q_m$) and momentum observables ($P_1, \dots, P_m$)  commute, generalizing Heisenberg's Uncertainty Principle for $m=1$ \cite{Dirac}. These days Dirac's algebra is now known as the {\it $m$-th Weyl algebra} $A_m(\Bbbk)$, which has $\Bbbk$-algebra presentation:
\begin{equation} \label{eq:Amk}
A_m(\Bbbk) = \frac{\Bbbk \langle x_1, \dots, x_m, y_1, \dots, y_m \rangle}{(x_ix_j - x_jx_i, \quad y_iy_j - y_jy_i, \quad y_ix_j - x_jy_i -\delta_{i,j})}.
\end{equation}
Here, $\delta_{i,j}$ is the Kronecker delta, and the generators $x_i$ and $y_i$ are viewed as elements of $\text{End}(\Bbbk[x_1, \dots, x_m])$ given resp. by multiplication by $x_i$ and partial derivation $\frac{\partial}{\partial x_i}$.

\medskip
Want more physics?? We're in luck-- the representation theory of numerous noncommutative $\Bbbk$-algebras play a vital role in several fields of physics. {\it Some} of these algebras and a physical area in which they appear are listed below. Happy exploring!

\bigskip

\begin{center}
{\small
\begin{tabular}{|c|c|c|}
\hline
\hspace{.1in} Noncommutative $\Bbbk$-Algebras \hspace{.1in} &\hspace{.1in}  Appearance in Physics \hspace{.1in} & \hspace{.1in} Reference (Year)\hspace{.1in} \\
\hline
\hline
 &&\\[-1.2em]
$\mathcal{W}$-{\it algebras} \quad &\quad  Conformal Field Theory \quad & \cite{BS} (1993) \\
  &&\\[-1.2em] 
 \hline 
  &&\\[-1.2em] 
\; {\it 4-dimensional Sklyanin algebras} \; \quad &\quad  Statistical Mechanics\quad & \cite{Sklyanin} (1982) \\
  &&\\[-1.2em] 
 \hline 
 &&\\[-1.2em]
 \; {\it 3-dimensional Sklyanin algebras} \;  \quad &\quad  String Theory \; \quad & \cite{BJL} (2000) \\
  &&\\[-1.2em] 
 \hline 
 &&\\[-1.2em]
  {\it Yang-Mills algebras} \quad &\quad  Gauge Theory \quad & \cite{C-DV} (2002)\\
  &&\\[-1.2em] 
 \hline 
 &&\\[-1.2em]
 {\it Superpotential algebras} \quad &\quad String Theory \quad & \cite{FHVWK} (2006)\\
  &&\\[-1.2em] 
 \hline 
 &&\\[-1.2em]
\; \begin{tabular}{c}
{\it Various Enveloping Algebras}\\
{\it of Lie algebras}
\end{tabular} \; \quad &\quad  * Everywhere * \quad &
Too many to list!
\\   [-1.2em]
  &&\\
 \hline
\end{tabular}
}
\end{center}

%\red{[FINISH- start by using ugrad talk and Coutinho's book intro. Define or highlight Symmetry, Representations, or Deformations]}

\section{Quantum Groups (1980s - 1990s) and Quantum Symmetries}
\label{sec:qgroups}

Let's begin here with a question mentioned in the introduction on the ties between symmetries [Definition~\ref{def:sym}]  and deformations [Definition~\ref{def:deform}].

\begin{question} \label{ques:sym-def}
How do we best handle (i.e., axiomatize, or ``make mathematical", the concept of) symmetries of deformations?
\end{question}

Several answers to this question lead us to use {\it Hopf algebras} [Definition~\ref{def:Hopf}]. But before we give the precise definition of this structure, we point out that Hopf algebras became prominent in mathematics in a few waves, including: its origins in Algebraic Topology \cite{Hopf}, role in Combinatorics \cite{JoniRota}, and abstraction in Category Theory \cite{JoyalStreet}. One tie to Noncommutative Algebra (in the context of Question~\ref{ques:sym-def}) first appeared in the 1980s in statistical mechanics, especially in the {\it Quantum Inverse Scattering Method} for solving quantum integrable systems. The Hopf algebras that arose this way were coined  {\it Quantum Groups} by Vladimir Drinfel'd  \cite{Drinfeld}, and have been a key structure in Noncommutative Algebra and physics ever~since.

\smallskip 

Instead of delving further into historical details, let's now discuss (quantum) symmetries of (deformed) algebras through concrete examples. Fix a field $\Bbbk$, and
recall from Figure~5 that an associative $\Bbbk$-algebra is a  $\Bbbk$-vector space equipped with the structure of a (unital) ring; we consider their deformations below.

\begin{definition}
Fix a $\Bbbk$-algebra $A$. 
A $\Bbbk$-algebra $A_{\text{def}}$ is a {\it deformation} of $A$ if $A_{\text{def}}$ and $A$ are the same as $\Bbbk$-vector spaces, 
but their respective multiplication rules are not necessarily the same. 
\end{definition}

\begin{example} Our running example  of a $\Bbbk$-algebra throughout this section will be the {\it $q$-polynomial algebra}:
$$\Bbbk_q[x,y] = \Bbbk \langle x,y \rangle/ (yx - qxy), \quad \text{for $q \in \Bbbk^\times$},$$
which is the quotient algebra of the free algebra $\Bbbk \langle x, y \rangle$ by the ideal $(yx - qxy)$.
Loosely speaking, $\Bbbk_q[x,y]$ is a $q$-deformation of $\Bbbk[x,y]$ as the former structure `approaches' the latter as $q \to 1$. More explicitly, note that $\Bbbk_q[x,y]$ and $\Bbbk[x,y]$ have the same $\Bbbk$-vector space basis $\{x^i y^j\}_{i,j \geq 0}$, but their multiplication rules differ for $q\neq 1$.
\end{example}

Now we let us examine symmetries of $\Bbbk_q[x,y]$ for $q\neq 1$ versus those of $\Bbbk[x,y]$. For this it is enough to consider {\it degree-preserving symmetries}, i.e., invertible transformations that send the generators $x$ and $y$ to a linear combination of themselves. Namely, let $V = \Bbbk x \oplus \Bbbk y$ be the generating space of $\Bbbk_q[x,y]$ (or $\Bbbk[x,y]$ with $q=1$). We want to pin down which invertible matrices in $GL(V) = GL_2(\Bbbk)$ also induce a symmetry of $\Bbbk_q[x,y]$, and to do so, we need to rewrite $\Bbbk_q[x,y]$ using the notion below. (From now on, we need an understanding of tensor products $\otimes$ and  a nice discussion of this operation can be found in \cite{Conrad2}.)

\begin{definition}
Given a $\Bbbk$-vector space $V$, the {\it tensor algebra} $T(V)$ is the $\Bbbk$-vector space $\bigoplus_{i \geq 0} V^{\otimes i}$ where $V^0 = \Bbbk$, and with multiplication given by concatenation, i.e.,  
$(v_1 \otimes \cdots \otimes v_m)(v_{m+1} \otimes \cdots \otimes v_{m+n}) 
= v_1 \otimes \cdots \otimes v_{m+n}.$
\end{definition}

Ideals $I$ of tensor algebras $T(V)$ are defined as usual, and one can define a quotient $\Bbbk$-algebra given by $T(V)/I$.

\begin{example} The free algebra $\Bbbk \langle x,y \rangle$ is identified with the tensor algebra $T(V)$ on the $\Bbbk$-vector space $V = \Bbbk x \oplus \Bbbk y$: for the forward direction insert  $\otimes$ between variables, and conversely suppress $\otimes$ between variables. The $q$-polynomial algebra $\Bbbk_q[x,y]$ is then identified as the quotient algebra of $T(\Bbbk x \oplus \Bbbk y)$ by the ideal $(y \otimes x - q x \otimes y)$.
%which as a $\Bbbk$-vector space is 
%$$(y \otimes x - q x \otimes y) = \bigoplus_{s,s' \in T(V)} s \otimes y \otimes x \otimes s' -  qs \otimes x \otimes y \otimes s'.$$
\end{example}

Now take $g \in GL(V)$ for $V = \Bbbk x \oplus \Bbbk y$. We want to extend this symmetry on $V$ to a symmetry of $\Bbbk_q[x,y]$ identified as $T(V)/(y \otimes x - q x \otimes y)$. Let's assume that, as in the case for group actions, $g$ acts on $T(V)$ diagonally:
\begin{equation} \label{eq:diag}
g( v \otimes v') := g(v) \otimes g(v'), \quad \forall v,v' \in V.
\end{equation}
Now the question is: When is the ideal $(y \otimes x - q x \otimes y)$ preserved under this action? In fact it suffices to show that
\begin{equation} \label{eq:ideal-pres}
g( y \otimes x - q x \otimes y) = \lambda(y \otimes x - q x \otimes y), \quad \text{for some } \lambda \in \Bbbk,
\end{equation}
since the $g$-action is degree preserving. To be concrete, say $g \in GL(V)$ is given by
\begin{equation} \label{eq:g-conc}
g(x) = \alpha x + \beta y \quad \text{ and } \quad  g(y) = \gamma x + \delta y, \quad \quad \text{for some } \alpha, \beta, \gamma, \delta \in \Bbbk.
\end{equation}
Then, $g( y \otimes x - q x \otimes y) = [g(y) \otimes g(x)] - q [g(x) \otimes g(y)]$, which is equal to
\[
\begin{array}{c}
%g( y \otimes x - q x \otimes y) = [g(y) \otimes g(x)] - q [g(x) \otimes g(y)]\\
%&= [(\gamma x + \delta y) \otimes (\alpha x + \beta y)] - q[(\alpha x + \beta y) \otimes (\gamma x + \delta y)]\\
%\; = 
(1-q)\alpha \gamma (x \otimes x) + (\beta \gamma - q\alpha \delta) (x \otimes y)
+ (\alpha \delta - q\beta \gamma)(y \otimes x) + (1-q)\beta \delta (y \otimes y).
\end{array}
\]
Therefore, the condition \eqref{eq:ideal-pres} is satisfied:

\begin{tabular}{ll}
$\bullet$ \; Always, &\quad  if $q=1$;\\
$\bullet$ \; Only when $\alpha = \delta = 0$ \; or \;  $\beta = \gamma = 0$, \hspace{.5in} &\quad if $q = -1$;\\
$\bullet$ \; Only when  $\beta = \gamma = 0$, &\quad if $q \neq \pm1$.
\end{tabular}

\noindent (Note that in the first case   $\lambda = \alpha \delta - \beta \gamma$, the determinant of $g$ when in matrix form.)

\smallskip
So, when we pass from the commutative polynomial algebra $\Bbbk[x,y]$ to its noncommutative deformation $\Bbbk_q[x,y]$ for $q \neq 1$, the amount of its degree-preserving symmetries shrinks abruptly. This is rather unsatisfying as passing ``continuously" from $\Bbbk[x,y]$ to $\Bbbk_q[x,y]$ does not yield a ``continuous passage" between their respective degree-preserving automorphism groups.  
%As mentioned in the introduction, this is due to the fact \red{[???]} groups do not admit deformations \red{[Ref]}.

\smallskip
We need to think beyond group actions like those in \eqref{eq:diag}. In general, we want to construct symmetries of a $\Bbbk$-algebra $T(V)/I$ by (i) considering symmetries of the generating space $V$, (ii) extending those to symmetries of $T(V)$, and then (iii) determining which symmetries in (ii) descend to $T(V)/I$. For step (i), take $V$ to be a representation of an algebraic object $H$, e.g., $H$ could be a group or a $\Bbbk$-algebra. (We often swap back and forth between using ``representations" and ``modules".) For (ii), one needs to tackle the issue of  building a direct sum and tensor product of $H$-representations. The former is  pretty straight-forward-- one can always construct the direct sum of $H$-representations to get another (the first guess is most likely the correct one!). 
But  if we are given two vector spaces $V_1$ and $V_2$ that are  $H$-modules, 

\begin{question} \label{ques:tens}
When is $V_1 \otimes V_2$ an $H$-module? \footnote{In categorical language, this is the question of whether the category of $H$-modules (or of representations of $H$) has a {\it monoidal} structure.}
\end{question}

 If $H$ were a group $G$, then one can give $V_1 \otimes V_2$ the structure of a (left) $G$-module via \eqref{eq:diag}.
We can extend this linearly to get  that $V_1 \otimes V_2$ is a module over a group algebra on $G$. But if $H$ were an arbitrary algebra, then the diagonal action on $V_1~\otimes~V_2$ does not necessarily give it the structure of an $H$-module (as we will see in Remark~\ref{rem:diag}). In fact, to have an action of $H$ on $V_1 \otimes V_2$ we first need \underline{algebra maps}
\[
 \Delta: H \to H \otimes H, \quad \Delta(h) \mapsto \sum h_1 \otimes h_2,\quad \quad  \text{and } \quad \epsilon: H \to \Bbbk.
\]  
Here, we use the {\it Sweedler notation} shorthand to denote elements of $\Delta(H)$. These maps should be compatible in a way that is dual to the manner that the multiplication map $m: H \otimes H \to H$ and unit map $\eta: \Bbbk \to H$ of an algebra are compatible (cf. $m(\eta \otimes \text{id}_H) = \text{id}_H = m(\text{id}_H \otimes \eta)$). That is, after identifying $\Bbbk \otimes H = H =  H \otimes \Bbbk$, 
\begin{equation} \label{eq:coalg}
(\epsilon \otimes \text{id}_H) \circ \Delta = \text{id}_H = (\text{id}_H \otimes \epsilon) \circ \Delta.
\end{equation}

\begin{definition} \cite[Chapter~5]{Radford}
An associative $\Bbbk$-algebra $H = (H, m, \eta)$ is  a $\Bbbk$-{\it bialgebra} if it equipped with algebra maps $\Delta$ ({\it coproduct}) and $\epsilon$ ({\it counit}), so that $(H, \Delta, \epsilon)$ is a {\it coassociative $\Bbbk$-coalgebra}  with the structures $(H, m, \eta)$ and $(H, \Delta, \epsilon)$ being compatible.
\end{definition}

To answer Question~\ref{ques:tens}: If $H$ is a bialgebra, the $H$-module structure on $V_1 \otimes V_2$ is $$\hspace{.7in} h(v_1 \otimes v_2) =:\sum h_1(v_1) \otimes h_2(v_2) \quad \quad \forall h \in H \text{ and } v_1,v_2 \in V.$$ We also get that $\Bbbk$ admits the structure of a {\it trivial} $H$-module via $h(1_\Bbbk) = \epsilon(h) 1_\Bbbk$.

\begin{remark}\label{rem:diag}
We cannot always use a diagonal action-- sometimes a fancier coproduct is needed to address Question~\ref{ques:tens}. To see this, take $H$ to be the 2-dimensional associative $\Bbbk$-algebra $\Bbbk[h]/(h^2)$ (e.g., so that we are considering  linear operators that are the zero map when composed with itself). If the coproduct of $H$ is $\Delta(h) = h \otimes h$, then $\epsilon(h) = 1$ by \eqref{eq:coalg}. But this implies
$0 = \epsilon(h^2) = \epsilon(h)^2 = 1$, a contradiction. To ``fix" this, check that the coproduct $\Delta(h) = h \otimes 1 + 1 \otimes h$ with the counit $\epsilon(h) = 0$ gives $\Bbbk[h]/(h^2)$ the structure of a bialgebra over $\Bbbk$.
\end{remark}

Moreover, one may be interested (in symmetries of) an algebra with generating space $V^*$, the {\it linear dual}; this will play a role later in Section~5.2. To get this, we want $V^*$ to have the induced structure of an $H$-module, and we need  an \underline{anti-algebra-automorphism} $S:H \to H$ of $H$ to proceed.

\begin{definition} \label{def:Hopf} \cite[Chapters 6-7]{Radford}
A $\Bbbk$-bialgebra $H = (H, m, \eta, \Delta, \epsilon)$ is a {\it Hopf algebra} over $\Bbbk$ if there exists anti-automorphism $S:H \to H$ ({\it antipode}) so that 
$$m \circ (S \otimes \text{id}_H) \circ \Delta ~=~ \eta \circ \epsilon ~=~ m \circ (\text{id}_H \otimes S) \circ \Delta.$$
\end{definition}

If $H$ is a Hopf algebra with $H$-module $V$, an action of $H$ on $V^*$ can be given~by \footnote{In this case, the category of $H$-modules is a {\it rigid} monoidal category.} 
$$\hspace{.7in} [h(f)](v) = f[S(h)(v)], \quad \quad \forall h \in H, \; f \in V^*, \; v \in V.$$

\begin{wrapfigure}{r}{0.4\textwidth}
\vspace{-.15in}
\centering \includegraphics[width=0.4\textwidth]{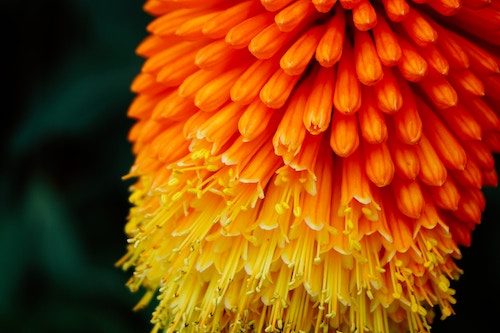}
{\small \caption{Symmetries deforming}}
\vspace{-.2in}
\end{wrapfigure}
Examples of Hopf algebras are {\it group algebras} on finite groups $\Bbbk G$, {\it function algebras} on algebraic groups $\mathcal{O}(G)$, and {\it universal enveloping algebras} of Lie algebras $U(\mathfrak{g})$, which are all considered  ``classical" in the sense that they are {\it commutative} (as an algebra,  $m\circ \tau = m$) or {\it cocommutative} (as a coalgebra, $\tau \circ \Delta = \Delta$), for $\tau(a \otimes b)  = b \otimes a$. Indeed, these Hopf algebras capture the actions of a group on a $\Bbbk$-algebra by automorphism and actions of a Lie algebra on a $\Bbbk$-algebra by derivation. Moreover, deformations (or {\it quantized} versions) of these structures provide a setting to handle  deformations of the aforementioned symmetries (cf. Question~\ref{ques:sym-def}); refer to \cite{AST, Jimbo, Lusztig, Manin} for examples of Hopf algebras arising in this fashion. We also recommend the excellent text on (actions of) Hopf algebras by Susan Montgomery~\cite{Montgomery}.

\smallskip

Now we summarize a few frameworks for studying (quantum) symmetries of a $\Bbbk$-algebra $A$ involving a group $G$ or  a Hopf algebra $H$. See \cite{Montgomery} for more details. \footnote{For more settings of quantum symmetry, see, e.g., \cite[Chapter 11]{Radford} for a categorical~framework.}

\begin{itemize}
\item[\textsc{[G-act]}] \quad \quad \quad  {\it Group actions} on $A$: That is, $A$ is a $G$-module with $G$-action map \linebreak $G \times A \to A$ given by $(g,a) \mapsto g(a)$ satisfying $g(ab) = g(a) \; g(b)$ and $g(1_A) = 1_A$, for all $g \in G$ and $a, b \in A$.

\medskip

\item[\textsc{[G-grd]}] \quad \quad \quad  {\it Group gradings} on $A$: That is, $A$ is $G$-graded if $A = \oplus_{g \in G} ~A_g$, for $A_g$ a $\Bbbk$-vector space, with $A_g \cdot A_h \subset A_{gh}$. When $G$ is finite, this is equivalent to $A$ being acted upon by the {\it dual group algebra} $(\Bbbk G)^*$.

\medskip

\item[\textsc{[H-act]}] \quad \quad \quad  {\it Hopf algebra} or {\it bialgebra actions} on $A$: That is, $A$ is an $H$-module  with $H$-action map $H \times A \to A$ given by $(h,a) \mapsto h(a)$ with $h(ab) = \sum h_1(a) \; h_2(b)$ and $h(1_A) = \epsilon(h) 1_A$, for all $h \in H$ and $a, b \in A$, with $\Delta(h) = \sum h_1 \otimes h_2$. 
\end{itemize}

\smallskip
Finally we end with an example of a Hopf algebra action on $\Bbbk_q[x,y]$, illustrating a scenario where Question~\ref{ques:sym-def}  has a possible answer. For other (more general) examples in the literature, we refer to \cite[Sections~IV.7 and VII.3]{Kassel}.

\begin{example} (A simplified version of \cite[Theorem~VII.3.3]{Kassel}) For ease, we take $\Bbbk$ to be $\mathbb{C}$. Also, let $q$ be a nonzero complex number that's not a root of unity. We aim  to produce an action of a Hopf algebra $H_q$ over $\mathbb{C}$ (whose structure depends on $q$) on the $q$-polynomial algebra $\mathbb{C}_q[x,y] = \mathbb{C} \langle x,y \rangle/ (yx - qxy)$, so that 
\begin{itemize}
\item the ``limit" of $H_q$ as $q \to 1$ is a ``classical" Hopf algebra $H$ (i.e., $H$ is either commutative or cocommutative and $H_q$ is a {\it $q$-deformation} of $H$), and

\smallskip
 
\item the ``limit" of the $H_q$-action on $\mathbb{C}_q[x,y]$ as $q \to 1$ is an action of $H$ on $\mathbb{C}[x,y]$.
\end{itemize}

%~= ~q^{\ell - 1} +  q^{\ell - 3} + \dots +  q^{- \ell +3 }  +  q^{- \ell +1 }.$
%We also define {\it $q$-differential operators} on the underlying $\mathbb{C}$-vector space of $\mathbb{C}_q[x,y]$ (being $\bigoplus_{i, j \geq 0} x^i y^j$) given by:
%$$\sigma_x(x^i y^j) = q^ix^iy^j; \; \;  \; \;  \sigma_y(x^i y^j) = q^jx^iy^j;  \; \;  
%%\frac{\partial_q(x^iy^j) }{\partial x}= [i]_q x^{i-1}y^j; 
% \; \;  \frac{\partial_q(x^iy^j)}{\partial y}= [j]_q x^i y^{j-1}.$$
%
We begin by defining a Hopf algebra $H_q$ with algebra presentation,
$$H_q = \mathbb{C}\langle g, g^{-1}, h \rangle/(gg^{-1} - 1, \; g^{-1}g -1, \; gh - q^2 hg),$$ along with coproduct, counit, and antipode given by
\[
\begin{array}{c}
\Delta(g) = g \otimes g, \quad \Delta(g^{-1}) = g^{-1} \otimes g^{-1}, \quad \Delta(h) = 1 \otimes h + h \otimes g,\\
\epsilon(g) = 1, \quad \epsilon(g^{-1}) = 1, \quad \epsilon(h) = 0, \quad S(g) = g^{-1}, \quad S(g^{-1}) = g, \quad S(h) = -hg^{-1}.
\end{array}
\]

Next we define a {\it $q$-number} 
$[\ell]_q  ~:=~ \frac{q^\ell - q^{-\ell}}{q - q^{-1}} $ for any integer $\ell$.
Now for any element $p = \sum_{i,j \geq 0} \lambda_{ij} x^iy^j$ in $\mathbb{C}_q[x,y]$, the rule below gives us an action of  $H_q$ on $\mathbb{C}_q[x,y]$:
$$g(p) = \sum_{i,j \geq 0} \lambda_{ij} q^{i-j} x^iy^j, \quad g^{-1}(p) = \sum_{i,j \geq 0} \lambda_{ij} q^{j-i} x^iy^j, 
\quad h(p) =  \sum_{i,j \geq 0} \lambda_{ij} [j]_q \;x^{i+1}y^{j-1}.$$
To check this, it suffices to show that (i) the relations of $H_q$ act  on $\mathbb{C}_q[x,y]$ by zero, and that (ii) the relation space of $\mathbb{C}_q[x,y]$ is preserved under the rule above. We'll provide some details here and leave the rest as an exercise. We compute:
\[ 
\begin{array}{lrl}
\text{For (i),}    &  (gh - q^2hg)(p) &=~ g(\sum \lambda_{ij} [j]_q \; x^{i+1}y^{j-1}) - q^2 h(\sum \lambda_{ij} q^{i-j} x^iy^j)\\
 \smallskip
 &&=~ \sum \lambda_{ij} [j]_q \;q^{i-j+2} x^{i+1}y^{j-1} - q^2 \sum \lambda_{ij} [j]_q \;q^{i-j} x^{i+1}y^{j-1} ~= 0;\\

 \text{For (ii),}  &   \; h(yx - qxy) &=~ [1(y)\; h(x) + h(y)\; g(x)] - q h(xy) ~=~ (x)(qx) - q(x^2) ~=0. 
 \end{array}
\]

Now the ``limit" of $H_q$ as $q \to 1$ is $H = \mathbb{C}[x] \otimes \mathbb{C} \mathbb{Z}$, the {\it tensor product of Hopf algebras},  for $\mathbb{Z} = \langle g \rangle$ (see, e.g., \cite[Exercise~2.1.19]{Radford}); $H$ is both commutative and cocommutative. Also, $H_q = H$ as $\mathbb{C}$-vector spaces. Moreover, as $q \to 1$, the generators $g$ and $g^{-1}$ (resp., $h$) of $H$ act on $\mathbb{C}_q[x,y]$ as the identity (resp., by $\frac{\partial}{\partial_y}$).

\end{example}

%%%%%%%%%%%%%%%%%%%%%%%%%%%%%%%%%%
%%%%%%%%%%%%%%%%%%%%%%%%%%%%%%%%%%
%%%%%%%%%%%%%%%%%%%%%%%%%%%%%%%%%%

\section{Research directions in Noncommutative Algebra}
\label{sec:research}

We highlight a couple of directions for research in Noncommutative Algebra in this section, building on the discussions of Sections~1-4. The material below could also serve as a topic for an undergraduate or Master's thesis project, or as a reading course topic. Finding a friendly faculty (or advanced graduate student) mentor to help with these pursuits is a good place to start...

%%%%%%%%%%%%%%%%%%%%%%%%%%%%%%%%%%
%%%%%%%%%%%%%%%%%%%%%%%%%%%%%%%%%%
%%%%%%%%%%%%%%%%%%%%%%%%%%%%%%%%%%

\subsection{On Symmetries} 

Continuing the discussions of Section~2 and~4, we propose the following avenue for research: Study of the symmetries of (algebraic structures that generalize) Hamilton's quaternions $\mathbb{H}$ [Problem~1]. One such generalization is given below.

\begin{definition} \label{def:quat} \cite[Section~5.4]{Cohn}
Fix a field $\Bbbk$ with char $\Bbbk \neq 2$, with nonzero scalars $a,b \in \Bbbk$. Then a {\it quaternion algebra} $Q(a,b)_\Bbbk$ is a $\Bbbk$-algebra that has an underlying 4-dimensional $\Bbbk$-vector space with basis $\{1, i ,j, k\}$, subject to multiplication rules
$$i^2 = a, \quad j^2 = b, \quad ij = -ji = k.$$
Note that $k^2 = ijk = -ab$, for instance. 
\end{definition}

Sometimes $Q(a,b)_\Bbbk$ is denoted by $(a,b)_\Bbbk$, by  $(a,b; \Bbbk)$, or even by $(a,b)$ if $\Bbbk$ is understood. The structure above extends the construction of Hamilton's quaternions [Definition~\ref{def:H}], namely $\mathbb{H} = Q(-1, -1)_\mathbb{R}$. Moreover, {\it split-quaternions}, $Q(-1, +1)_\mathbb{R}$, also appear frequently in the literature.

\smallskip
Fun fact: A quaternion algebra is either a 4-dimensional $\Bbbk$-division algebra [Definition~\ref{def:div}], or is isomorphic to the matrix algebra $M_2(\Bbbk)$! (The latter is called the {\it split} case.) Also, these cases are characterized by the {\it norm} of elements $Q(a,b)_\Bbbk$:  
$$N(a_0 + a_1 i +a_2 j +a_3 k):=a_0^2 - a a_1^2 - ba_2^2 + ab a_3^2, \;\; \text{ for }\;  a_0, a_1, a_2, a_3 \in \Bbbk.$$
Namely, if $\Bbbk$ has characteristic not equal to 2, then $Q(a,b)_\Bbbk$ is a division algebra precisely when $N(a_0 + a_1 i +a_2 j +a_3 k) = 0$ only for $(a_0, a_1, a_2, a_3) = (0, 0, 0, 0)$ \cite[Proposition~5.4.3]{CE}. For instance, $\mathbb{H}=Q(-1, -1)_\mathbb{R}$ is a $\mathbb{R}$-division algebra since
$$N(a_0 + a_1 i +a_2 j +a_3 k)=a_0^2 + a_1^2 + a_2^2 + a_3^2$$ for $a_0, a_1, a_2, a_3 \in \mathbb{R}$, and
is 0 if and only if $(a_0, a_1, a_2, a_3) = (0, 0, 0, 0)$. 

\smallskip
Quaternion algebras (in the generality of Definition~\ref{def:quat} above) have appeared primarily in number theory \cite{Vigneras} \cite[Chapter~5]{Miyake} and in the study of quadratic forms \cite[Chapter~III]{Lam:qf}. They have also been used in hyperbolic geometry \cite{Macfarlane} \cite[Chapter~2]{MR}, and in various parts of physics and engineering; see, e.g., \cite{Baylis} and \cite{Onsager}. For more details about their applications and structure, see \cite{Conrad} and the references within.

\smallskip Recall from Section~4 that there are several frameworks for studying symmetries of a $\Bbbk$-algebra, including group actions \textsc{[G-act]}, group gradings \textsc{[G-grd]}, and  Hopf algebra actions \textsc{[H-act]}. Also, the latter symmetries are considered to be {\it quantum symmetries} if $H$ is non(co)commutative, as  discussed by Figure~14. 

\begin{problem}
Study the (quantum) symmetries of quaternion algebras. Namely, pick a setting \textsc{[G-act]}, \textsc{[G-grd]}, \textsc{[H-act]}, a collection of structures ($G$ or $H$) in this class, and classify all such symmetries of $G$ or $H$ on $Q(a,b)_\Bbbk$. 
\end{problem}

Even if this problem is not addressed in full generality, a collection of examples would be quite useful for the literature.  For instance, a group grading of $Q(-1,-1)_\mathbb{R} = \mathbb{H}$ was used in recent work of Cuadra and Etingof as a counterexample to show that their main result on {\it faithful} group gradings on division algebras fails when the ground field is not algebraically closed \cite[Theorem~3.1, Example~3.4]{CE}. 
%Also, an example of a $Q_8$-grading on a quaternion algebra $Q$ over an algebraically closed field of characteristic 0 was produced to illustrate a non-abelian group grading on such structures such $Q$ \cite[Example 2.10]{CE}. Here, $Q_8$ is the {\it quaternion group}, which by definition, is the non-abelian subgroup $\{\pm 1, \pm i, \pm j, \pm k\}$  of $\mathbb{H}$ (or, of $Q(a,b)_\Bbbk$) under multiplication.

\smallskip

There are also other works that partially address Problem~1, such as on group gradings \cite{%AM, 
CM, MG-O1, MG-O2} and Hopf algebra (co)actions \cite{DT, VO-Z}. These papers also contain work on (quantum) symmetries of some generalizations of quaternion algebras; \textbf{Problem~1  can also be posed for these generalizations of} $Q(a,b)_\Bbbk$ as well.

\smallskip

Moreover, \textbf{a second part of Problem~1} could include the study of two algebraic structures formed by the symmetries constructed above, namely,  the {\it subalgebra of (co)invariants}, and the {\it smash product algebra} (or,  {\it skew group algebra} if \textsc{[G-act]} is used). See \cite{Montgomery} for the definitions, examples, and a discussion of various uses of these algebraic structures. Overall, after one gets comfortable with the terminology, such problems are computational in nature ... and  fun to do!

%%%%%%%%%%%%%%%%%%%%%%%%%%%%%%%%%%
%%%%%%%%%%%%%%%%%%%%%%%%%%%%%%%%%%
%%%%%%%%%%%%%%%%%%%%%%%%%%%%%%%%%%

\subsection{On Representations}  \label{sec:researchrep} 

In this section, $\Bbbk$ is a field of characteristic zero.

\smallskip
Towards a research direction in representation theory (continuing the discussion in Section~3) it is natural to think further about the representations of the first Weyl algebra $A_1(\Bbbk)$. Since there are no finite-dimensional representations of $A_1(\Bbbk)$ [Proposition~\ref{prop:A1-inf}], what are its infinite-dimensional representations?
To get one for example, identify $A_1(\Bbbk)$ as a ring of differential operators on $\Bbbk[x]$ where the generators $x$ and $y$ act as multiplication by $x$ and by $\frac{d}{dx}$, respectively. So, by fixing a basis $\{1, x, x^2, x^3, \dots\}$ of $\Bbbk[x]$, we get the (matrix form of) the infinite-dimensional representation in ~\eqref{eq:A1k-inf}.  Producing explicit infinite-dimensional representations of $A_1(\Bbbk)$ is tough in general. But there are many works on the {\it abstract} representation theory of $A_1(\Bbbk)$ and of other rings of differential operators, and we recommend the student-friendly text of S.C. Coutinho on {\it algebraic $D$-modules} \cite{Coutinho} for more information. 

\smallskip
Now for a concrete research problem to pursue, we suggest working with deformations of Weyl algebras instead, particularly those that admit finite-dimensional representations (as this is more feasible computationally). One could:

\begin{problem} \label{prob:qwa}
Examine the (explicit) representation theory of {\it quantum Weyl algebras} ({\it at roots of unity}) [Definition~\ref{def:qwa}].
\end{problem}

Before we discuss quantum Weyl algebras, we introduce some terminology  that will be of use later in order to make the problem above more precise. The text \cite{Etingofetal} (which, again, is student-friendly) is a nice reference for more details.

\begin{definition} Take a $\Bbbk$-algebra $A$ with a representation $$\phi: A \to \text{Mat}_n(\Bbbk) \; (\cong \text{End}(V)) \quad \text{ for }  \; V = \Bbbk^{\oplus n}.$$ 
\begin{enumerate}
\item We say that $\phi$ is {\it decomposable} if we can decompose $V$ as $W_1 \oplus W_2$ with $W_1, W_2 \neq 0$ so that $\phi|_{W_k}: A \to \text{End}(W_k)$ are representations of $A$ for $k =1,2$. 
Otherwise, we say that $\phi$ is {\it indecomposable}.
\smallskip
\item The representation $\phi$ is {\it reducible} if there exists a proper subspace $W$ of $V$ so that $\phi|_{W}: A \to \text{End}(W)$ is a representation of $A$; here, $\phi|_W$ is called a {\it (proper) subrepresentation} of $\phi$. If $\phi$ does not have any proper subrepresentations, then  $\phi$ is {\it irreducible}; the corresponding $A$-module $V$ is said to be {\it simple}  (cf. Figure~12).

\smallskip
\item Take another representation $\phi': A \to \text{End}(V')$ of $A$. We say that $\phi'$ is {\it equivalent} (or {\it isomorphic}) to $\phi$ if $\dim V = \dim V'$ and there exists an invertible $\Bbbk$-linear map $\rho: V \to V'$ so that $\rho(\phi_a(v)) = \phi'_a(\rho(v))$ for all $a \in A$ and $v \in V$.
\end{enumerate}
\end{definition}

Irreducible representations are indecomposable; the converse doesn't always~hold. 

\smallskip 
To understand the notions above in terms of matrix solutions of equations (cf. Figure~13), take a {\it finitely presented} $\Bbbk$-algebra $A$, that is, $A$ has finitely many noncommuting variables $x_i$ as generators, and finitely many words $f_j(\underline{x})$ in $x_i$ as relations:
$$ A = \frac{\Bbbk \langle x_1, \dots, x_t \rangle}{\left(f_1(\underline{x}), \dots, f_r(\underline{x})\right)}.$$
Let us also fix an $n$-dimensional representation of $A$, given by
$$\phi: A \to \text{Mat}_n(\Bbbk), \quad x_i \mapsto X_i \; \; \text{ for } i = 1, \dots, t.$$

\begin{definition} \label{def:rep-prop} Retain the notation above. Suppose that we have a matrix solution $\underline{X} = (X_1, \dots, X_t)$ to the system of equations $f_1(\underline{x}) =  \dots = f_r(\underline{x}) = 0$.
\begin{enumerate}
\item If each matrix $X_i$ can be written as a direct sum of matrices $X_{i,1} \oplus X_{i,2}$, where 
\begin{itemize} 
\item  $X_{i,k} \in \text{Mat}_{n_k}(\Bbbk)$ with $k=1,2$ for some positive integers $n_1$ and $n_2$, and
\smallskip

\item  $\underline{X_k} = (X_{1,k}, \dots, X_{t,k})$ is a solution to  $f_1(\underline{x}) =  \dots = f_r(\underline{x}) = 0$ for $k=1,2$,
\end{itemize}
then the matrix solution $\underline{X}$ is {\it decomposable}. Otherwise, $\underline{X}$ is {\it indecomposable}.
\smallskip
\item For $\text{Mat}_n(\Bbbk)$ identified as $\text{End}(V)$ with $V = \Bbbk^{\oplus n}$, suppose that there exists a proper subspace $W$ of $V$ that is stable under the action of each $X_i$. Then  we say that $\underline{X}$ is {\it reducible}. Otherwise, $\underline{X}$ is {\it irreducible}.
\smallskip
\item We say that another matrix solution $\underline{X'} \in \text{Mat}_{n'}(\Bbbk)^{\times t}$ to the system of equations $f_1(\underline{x}) =  \dots = f_r(\underline{x}) = 0$ is {\it equivalent} (or {\it isomorphic}) to $\underline{X}$ if $n = n'$ and 
 there exists an invertible matrix $P \in \text{GL}_{n}(\Bbbk)$ so that $P\;X_i\;P^{-1} = X_i$ for all $i$.
\end{enumerate}
\end{definition}

So two representations of $A$ (or, two matrix solutions of $\{f_j(\underline{x}) = 0\}_{j=1}^r$) are equivalent precisely when they are the same up to change of basis of $V = \bigoplus_{i=1}^t x_i$. 
%For example, consider 
%the representation of $A_1(\Bbbk)$  in \eqref{eq:repWeyl} given by \eqref{eq:A1k-inf} -- it is irreducible and thus indecomposable. Moreover, by choosing a different basis of $V=\Bbbk[x]$ we can obtain an equivalent representation of $A_1(\Bbbk)$, and thus, matrix solution to the Normalized Fundamental Equation \eqref{eq:NFE}. 
Therefore Problem~\ref{prob:qwa} can be refined as follows.

\medskip

\noindent {\bf Precise version of Problem~\ref{prob:qwa}.} Classify the explicit irreducible representations of the quantum Weyl algebras [Definition~\ref{def:qwa}], up to equivalence.
\medskip

Let's define the quantum Weyl algebras now. One way of getting these algebras is by deforming the $m$-th Weyl algebras $A_m(\Bbbk)$ from \eqref{eq:Amk} via the symmetry discussed below.  (The reader may wish to skip to Definition~\ref{def:qwa} for the outcome of this chat.)

\begin{definition} Fix a $\Bbbk$-vector space $V$. 
\begin{enumerate}
\item A $\Bbbk$-linear transformation $c: V \otimes V \to V \otimes V$ is a  {\it braiding} if it satisfies the braid relation,
$(c \otimes \text{id}_V) \circ (\text{id}_V \otimes c) \circ (c \otimes \text{id}_V) 
=  (\text{id}_V \otimes c) \circ (c \otimes \text{id}_V)\circ  (\text{id}_V \otimes c) $
as maps $V^{\otimes 3} \to V^{\otimes 3}$.

\smallskip 

\item  A braiding $\mathcal{H}: V \otimes V \to V \otimes V$ is a {\it Hecke symmetry} if it satisfies the Hecke condition, 
$(\mathcal{H} - q\;\text{id}_{V \otimes V}) \circ (\mathcal{H} + q^{-1}\;\text{id}_{V \otimes V})~=~0$
as maps $V \otimes V \to V \otimes V$, for some nonzero $q \in \Bbbk$.
\end{enumerate}
\end{definition}

Given a Hecke symmetry $\mathcal{H} \in \text{End}(V \otimes V)$ one can form the {\it $\mathcal{H}$-symmetric algebra} 
$S_{\mathcal{H},q}(V) = T(V)/\left(\textnormal{Image}(\mathcal{H} - q\;\textnormal{id}_{V \otimes V})\right).$
For example, when $\mathcal{H} = {\sf flip}$ (sending $x_i \otimes x_j$ to $x_j \otimes x_i$) and $q=1$ we get that $S_{{\sf flip},1}(V)$ is the symmetric algebra $S(V)$ on $V$; this is isomorphic to the polynomial ring $\Bbbk[x_1, \dots, x_m]$ for $V = \bigoplus_{i=1}^m \Bbbk x_i$.

\smallskip
Summarizing the discussion in \cite{GZ}, we now build a $q$-version of a Weyl algebra using a Hecke symmetry $\mathcal{H}$ as follows. Consider the dual vector space $V^*$ and the induced $\Bbbk$-linear map $\mathcal{H}^* \in \text{End}(V^* \otimes V^*)$. %and build the symmetric algebra $S_{\mathcal{H}^*,q}(V^*)$. 
Then construct the algebra $A_{\mathcal{H},q}(V \oplus V^*)$ on $V \oplus V^*$, which is the tensor algebra $T(V \oplus V^*)$ subject to the relations: 
$\textnormal{Image}(\mathcal{H} - q\;\textnormal{id}_{V \otimes V})$, and $\textnormal{Image}(\mathcal{H}^* - q^{-1}\;\textnormal{id}_{V^* \otimes V^*})$, and certain relations entertwining generators from $V$ with those from $V^*$ by using $\mathcal{H}$. The resulting algebra $A_{\mathcal{H},q}(V \oplus V^*)$ is called the {\it quantum Weyl algebra associated to $\mathcal{H}$}. 

\smallskip
For simplicity, we provide the presentation of $A_{\mathcal{H},q}(V \oplus V^*)$ for the standard 1-parameter Hecke symmetry  given on \cite[page~442]{JingZhang} (provided in the form of an R-{\it matrix}). Here,
$V = \bigoplus_{i=1}^m \Bbbk x_i$ and $V^* = \bigoplus_{i=1}^m \Bbbk y_i$ with $y_i := x_i^*$ (linear dual of~$x_i$).

\begin{definition} \label{def:qwa} \cite[page~442]{JingZhang} \cite[Definition~1.4]{GZ} Take $m \geq 2$. The {\it 1-parameter quantum Weyl algebra} is an associative $\Bbbk$-algebra $A_m^q(\Bbbk)$ with noncommuting generators $x_1, \dots, x_m,$ $y_1, \dots, y_m$ subject to relations
\[
\begin{array}{rll}
x_ix_j &= qx_jx_i, \quad y_iy_j = q^{-1}y_jy_i, &\forall i<j\\
y_ix_j &= qx_jy_i, &\forall i \neq j\\
y_ix_i &= 1+ q^2x_i y_i + (q^2-1) \sum_{j>i}x_jy_j,\quad &  \forall i.
\end{array}
\]
By convention, we define $A_1^q(\Bbbk)$ to be $\Bbbk\langle x,y \rangle/(yx - qxy -1)$. If $q$ is a root of unity then we refer to these algebras as quantum Weyl algebras {\it at a root of unity}.
\end{definition}

Notice that  one gets the Weyl algebras $A_1(\Bbbk)$ [Definition~\ref{def:A1k}] and $A_m(\Bbbk)$ [Equation~\eqref{eq:Amk}] by taking the ``limit" of $A_1^q(\Bbbk)$ and $A_m^q(\Bbbk)$ as $q \to 1$, respectively.

\smallskip

Fun fact: If $q$ is a root of unity, say of order $\ell$, then all irreducible representations of a quantum Weyl algebra $A_{\mathcal{H},q}(V \oplus V^*)$ are finite-dimensional! Moreover in this case, the dimension of an irreducible representation is $A_{\mathcal{H},q}(V \oplus V^*)$ is bounded above by some positive integer $N(\ell)$ depending on $\ell$, and this bound is met most of the time. This is part of a general phenomenon for {\it quantum $\Bbbk$-algebras} with scalar parameters-- they have infinite-dimensional irreducible representations in the generic case, and in the root of unity case all of their irreducible representations are finite-dimensional. Further, in the root of unity case, most irreducible representations of a quantum algebra $A$ have dimension equal to the {\it polynomial identity (PI) degree} of $A$ (See, for instance, the informative text of Brown-Goodearl \cite{BG}). For example, the PI degree of $A_1^q(\Bbbk)$ is equal to $\ell$ when $q$ is a root of unity of order~$\ell$.

\smallskip

This leads us to discussion of \textbf{a partial answer to  Problem~\ref{prob:qwa}}. Indeed, one was achieved for $A_1^q(\Bbbk)$, for $q$ a root of unity of order $\ell$, in two undergraduate research projects directed by E. Letzter  \cite{BLTPS} and by L. Wang \cite{HW}. The explicit irreducible matrix solutions $(X, Y)$ to the equation $YX - qXY = 1$ were computed in these works (up to equivalence), the majority of which are $\ell$-by-$\ell$ matrices. 

\smallskip

Naturally,  the \textbf{next case for Problem~\ref{prob:qwa}} is the representation theory of quantum Weyl algebras $A_{\mathcal{H},q}(V \oplus V^*)$, where $\dim_\Bbbk V = 2$ and $q$ is a root of unity; this should build on the partial answer  above. There are a few routes one could take, such as examining $A_m^q(\Bbbk)$ for $m \geq 2$, or more generally, addressing Problem~2 for  {\it multi-parameter quantum Weyl algebras} as in \cite[Example~2.1]{GZ} \cite[Definition~1.2.6]{BG}. 

\smallskip

Why care? One reason is that quantum Weyl algebras have appeared in numerous works in mathematics and physics, including Deformation Theory \cite{GG, GZ, JingZhang, Jordan}, Knot Theory \cite{FT}, Category Theory \cite{Laugwitz}, Quantum mechanics and Hypergeometric Functions \cite{Spir} to name a few. Therefore, any (partial) resolution to Problem~2 would be a welcomed addition to the literature. So let's have a go at this. :)

%%%%%%%%%%%%%%%%%%%%%%%%%%%%%%%%%%%
%%%%%%%%%%%%%%%%%%%%%%%%%%%%%%%%%%%
%%%%%%%%%%%%%%%%%%%%%%%%%%%%%%%%%%%
%
%
%\subsection{On Deformations} 
%
%
%\red{[computing explicit Hopf 2-cocycles?? Need to find previous works. Use twisting paper of Guillot-Kassel-Masuoka-- has explicit examples of twisting under setting \textsc{[G-act]} and \textsc{[G-grd]} above. Maybe 0410263 for ``lazy cocycles". Admittedly this is going to be the most difficult project, requiring the most training.]}
%

%%%%%%%%%%%%%%%%%%%%%%%%%%%%%%%%%%
%%%%%%%%%%%%%%%%%%%%%%%%%%%%%%%%%%
%%%%%%%%%%%%%%%%%%%%%%%%%%%%%%%%%%

\begin{acknowledgement}
C. Walton is partially supported  by the US National Science Foundation with grants \#DMS-1663775 and 1903192, and with a research fellowship from the Alfred P. Sloan foundation. The author thanks the anonymous referees and Gene Abrams for their valuable feedback.
\end{acknowledgement}

\noindent \small{\bf Photo and Figure credits} \;  Figs. 2, 5-7, 9-10, 12-13: Author. (** = from unsplash.com)\\
\begin{tabular}{ll}
\; Fig. 1: Tammie Allen, $@$tammeallen**. \quad\quad& Fig. 8: Wikipedia, user: JP.  \\
\; Fig. 3: GazzaPax (flickr.com). & Fig. 11: Billy Huynh, $@$billy\_huy**.\\
\; Fig. 4:  Karolina Szczur, $@$thefoxis**. & Fig 14. Dan Gold, @danielcgold**.\\
\end{tabular}

\pagebreak
\section*{EDGE and Me.} 
I was first introduced to the EDGE program the summer before my first year of graduate school, as some of my mentors suggested that participating in the program would be a great way to build a support network before beginning my studies. At the time, I decided to go with other options to get prepared for graduate school but I kept EDGE in mind for future involvement. Fortunately, during the first year of my post-doctoral position I was granted the opportunity to teach for EDGE. It was fantastic to work with other women, including many women of color, who were about to embark on their graduate school journeys. I was also honored to have the opportunity to work with other faculty who `walk the walk' in efforts to increase diversity, inclusion, and equity of researchers and educators in the mathematical sciences. Fortunately I was able to participate as an EDGE instructor for the remainder of my post-doctoral years, and the sisterhood that the EDGE program has provided helped facilitate my path up the academia ladder. 
\smallskip

Being able to see myself in others --in students coming after me, in faculty clearing the path for me, and in peers with me along the way-- is a crucial part of my finding happiness and a sense of belonging in this job. This is especially true for women (of color) in general because there are many extra obstacles, major and minor, that we have to confront in order to succeed. 
For instance, here's an annoying one: During my literature search for this article I came quotes across like,
%for this article I noticed that there weren't many women whose work I would cite for the specific topics that I chose. So do I change topics, or do a good job of working with what I picked? (I do want others to know about the work of the women that I did cite after all.) Extra decisions. 
%Or consider quotes like, 
\begin{itemize}
\item[*] ``[...] developed by the Leningrad School (Ludwig Faddeev, Leon Takhtajan, Evgeny Sklyanin, Nicolai Reshetikhin and Vladimir Korepin) and related work by the Japanese School" [with no Japanese mathematicians listed], and
\smallskip

\item[*] ``My interview was finished when a dolled-up woman with butterfly-shaped glasses appeared, who informed me that I should rise because a lady has entered the room" [when this woman's appearance had nothing to do with the topic of the article and no other women were mentioned].
\end{itemize}
It certainly took extra energy to decide how to address these exclusionary passages (usually being `Don't be distracted by this mess') and keep moving. Those little, extra efforts add up over time.  
\smallskip

But what has kept me going?  Loving mathematics, and having a network of people like those in the EDGE program who love mathematics as well and view the field through a similar lens. It is my humble wish to help clear the path so that EDGE program participants and other marginalized folks can see themselves, not through the muddied lens of others' biases or prejudices, but with the proper view of using one's talents (mathematics) to find happiness, community, and fulfillment with this work. So when I receive email threads like, 
\begin{itemize}
\item[*] ``Please join me in congratulating two EDGErs on successfully completing their PhDs: Shanise Walker (E'12), who received her PhD in mathematics from Iowa State University in May and Jessica De Silva (E'13), who received her PhD from the University of Nebraska-Lincoln in June!
Congrats Dr. Walker! Congrats Dr. De Silva!" --T. Diercks (EDGE program~admin.),\\ followed by

\smallskip

\item[*] ``Wonderful news!!!!! Warmest congratulations, Shanise and Jessica, and all the best moving forward in your careers. Hugs, Rhonda" -- Rhonda Hughes (co-founder of the EDGE program);

\smallskip

\item[*] ``Congratulations, ladies!" -- Chassidy Bozeman (2012 EDGE program participant);

\smallskip

\item[*] ``BIG CONGRATULATIONS AND LOTS OF JOYFUL NOISE!!!!   Awesome.
I think we just passed 90 EDGE PhDs !! Bursting with admiration and pride...  
Ami" 
\\-- Ami Radunskaya (EDGE program co-director); 

%\smallskip
%
%\item[*] \red{``Many congratulations to both of you!!   How wonderful!
%I am so proud of all the EDGE accomplishments and the community you form. Sylvia"\\ -- Sylvia Bozeman (co-founder of the EDGE program), }\footnote{\red{Still need permission to use}}
\end{itemize}
it gives me extra energy to proceed, to not be distracted, and to keep moving. And those meaningful, inspirational boosts add up over time!

%\red{(recommended length: 250-1000 words) for each author with an EDGE
%connection that could be titled ?EDGE and Me.? This paragraph should be
%written in informal prose and should include your EDGE year(s), location(s),
%and role(s). If you were an EDGE participant, please also include the year and
%name of the doctoral institution that granted your doctorate (or anticipated
%doctorate). Then, offer the reader a meaningful take-away that speaks to the
%EDGE?s impact on you. This take-away may be as unique as you are. Perhaps
%you want to tell a meaningful EDGE-related story? Or perhaps you want to
%note how EDGE has had an impact on your professional work today? Or
%perhaps you want to thank a particular EDGE classmate, mentor, faculty
%member, or director for a very specific reason? Anything (well, almost
%anything) goes here as long as it will leave the reader thinking, ?Wow, the
%sum of these authors? stories makes me realize that EDGE is really something
%special.? If you would like help drafting this paragraph, please reach out to
%any editor. We will be happy to ask you some questions about EDGE that may
%help you brainstorm.}

\end{document}